\documentclass[11pt,twoside,a4paper]{article}
\usepackage{amsmath}
\usepackage{amssymb}
\usepackage{stmaryrd}
\usepackage[latin1]{inputenc}
\usepackage[T1]{fontenc}
\usepackage[english]{babel}
\newcommand{\tun}{\begin{picture}(5,0)(-2,-1)
\put(0,0){\circle*{2}}
\end{picture}}

\newcommand{\tdeux}{\begin{picture}(7,7)(0,-1)
\put(3,0){\circle*{2}}
\put(3,0){\line(0,1){5}}
\put(3,5){\circle*{2}}
\end{picture}}

\newcommand{\ttroisun}{\begin{picture}(15,12)(-5,-1)
\put(3,0){\circle*{2}}
\put(-0.65,0){$\vee$}
\put(6,7){\circle*{2}}
\put(0,7){\circle*{2}}
\end{picture}}
\newcommand{\ttroisdeux}{\begin{picture}(5,15)(-2,-1)
\put(0,0){\circle*{2}}
\put(0,0){\line(0,1){5}}
\put(0,5){\circle*{2}}
\put(0,5){\line(0,1){5}}
\put(0,10){\circle*{2}}
\end{picture}}

\newcommand{\tquatreun}{\begin{picture}(15,12)(-5,-1)
\put(3,0){\circle*{2}}
\put(-0.65,0){$\vee$}
\put(6,7){\circle*{2}}
\put(0,7){\circle*{2}}
\put(3,7){\circle*{2}}
\put(3,0){\line(0,1){7}}
\end{picture}}
\newcommand{\tquatredeux}{\begin{picture}(15,18)(-5,-1)
\put(3,0){\circle*{2}}
\put(-0.65,0){$\vee$}
\put(6,7){\circle*{2}}
\put(0,7){\circle*{2}}
\put(0,14){\circle*{2}}
\put(0,7){\line(0,1){7}}
\end{picture}}
\newcommand{\tquatretrois}{\begin{picture}(15,18)(-5,-1)
\put(3,0){\circle*{2}}
\put(-0.65,0){$\vee$}
\put(6,7){\circle*{2}}
\put(0,7){\circle*{2}}
\put(6,14){\circle*{2}}
\put(6,7){\line(0,1){7}}
\end{picture}}
\newcommand{\tquatrequatre}{\begin{picture}(15,18)(-5,-1)
\put(3,5){\circle*{2}}
\put(-0.65,5){$\vee$}
\put(6,12){\circle*{2}}
\put(0,12){\circle*{2}}
\put(3,0){\circle*{2}}
\put(3,0){\line(0,1){5}}
\end{picture}}
\newcommand{\tquatrecinq}{\begin{picture}(9,19)(-2,-1)
\put(0,0){\circle*{2}}
\put(0,0){\line(0,1){5}}
\put(0,5){\circle*{2}}
\put(0,5){\line(0,1){5}}
\put(0,10){\circle*{2}}
\put(0,10){\line(0,1){5}}
\put(0,15){\circle*{2}}
\end{picture}}

\newcommand{\tdun}[1]{\begin{picture}(10,5)(-2,-1)
\put(0,0){\circle*{2}}
\put(3,-2){\tiny #1}
\end{picture}}

\newcommand{\tddeux}[2]{\begin{picture}(12,5)(0,-1)
\put(3,0){\circle*{2}}
\put(3,0){\line(0,1){5}}
\put(3,5){\circle*{2}}
\put(6,-2){\tiny #1}
\put(6,3){\tiny #2}
\end{picture}}

\newcommand{\tdtroisun}[3]{\begin{picture}(20,12)(-5,-1)
\put(3,0){\circle*{2}}
\put(-0.65,0){$\vee$}
\put(6,7){\circle*{2}}
\put(0,7){\circle*{2}}
\put(5,-2){\tiny #1}
\put(9,5){\tiny #2}
\put(-5,5){\tiny #3}
\end{picture}}
\newcommand{\tdtroisdeux}[3]{\begin{picture}(12,15)(-2,-1)
\put(0,0){\circle*{2}}
\put(0,0){\line(0,1){5}}
\put(0,5){\circle*{2}}
\put(0,5){\line(0,1){5}}
\put(0,10){\circle*{2}}
\put(3,-2){\tiny #1}
\put(3,3){\tiny #2}
\put(3,9){\tiny #3}
\end{picture}}

\newcommand{\ptroisun}{\begin{picture}(15,12)(-5,-1)
\put(3,7){\circle*{2}}
\put(-0.65,0){$\wedge$}
\put(6,0){\circle*{2}}
\put(0,0){\circle*{2}}
\end{picture}}

\newcommand{\pquatreun}{\begin{picture}(15,12)(-5,-1)
\put(3,7){\circle*{2}}
\put(-0.65,0){$\wedge$}
\put(6,0){\circle*{2}}
\put(0,0){\circle*{2}}
\put(3,0){\circle*{2}}
\put(2.9,0){\line(0,1){7}}
\end{picture}}
\newcommand{\pquatredeux}{\begin{picture}(15,18)(-5,-1)
\put(3,14){\circle*{2}}
\put(-0.65,7){$\wedge$}
\put(6,7){\circle*{2}}
\put(0,7){\circle*{2}}
\put(0,0){\circle*{2}}
\put(0,0){\line(0,1){7}}
\end{picture}}
\newcommand{\pquatretrois}{\begin{picture}(15,18)(-5,-1)
\put(3,14){\circle*{2}}
\put(-0.65,7){$\wedge$}
\put(6,7){\circle*{2}}
\put(0,7){\circle*{2}}
\put(6,0){\circle*{2}}
\put(6,0){\line(0,1){7}}
\end{picture}}
\newcommand{\pquatrequatre}{\begin{picture}(15,18)(-5,-1)
\put(3,7){\circle*{2}}
\put(-0.65,0){$\wedge$}
\put(6,0){\circle*{2}}
\put(0,0){\circle*{2}}
\put(3,12){\circle*{2}}
\put(3,7){\line(0,1){5}}
\end{picture}}
\newcommand{\pquatrecinq}{\begin{picture}(15,12)(-5,-1)
\put(0,0){\circle*{2}}
\put(7,0){\circle*{2}}
\put(0,7){\circle*{2}}
\put(7,7){\circle*{2}}
\put(0,0){\line(0,1){7}}
\put(7,0){\line(0,1){7}}
\put(.5,1.5){$\scriptstyle \diagup$}
\end{picture}}
\newcommand{\pquatresix}{\begin{picture}(15,12)(-5,-1)
\put(0,0){\circle*{2}}
\put(7,0){\circle*{2}}
\put(0,7){\circle*{2}}
\put(7,7){\circle*{2}}
\put(0,0){\line(0,1){7}}
\put(7,0){\line(0,1){7}}
\put(0,1.5){$\scriptstyle \diagdown$}
\end{picture}}
\newcommand{\pquatresept}{\begin{picture}(15,12)(-5,-1)
\put(0,0){\circle*{2}}
\put(7,0){\circle*{2}}
\put(0,7){\circle*{2}}
\put(7,7){\circle*{2}}
\put(0,0){\line(0,1){7}}
\put(7,0){\line(0,1){7}}
\put(.5,1.5){$\scriptstyle \diagup$}
\put(0,1.5){$\scriptstyle \diagdown$}
\end{picture}}
\newcommand{\pquatrehuit}{\begin{picture}(15,18)(-5,-1)
\put(3,0){\circle*{2}}
\put(-0.65,0){$\vee$}
\put(6,7){\circle*{2}}
\put(0,7){\circle*{2}}
\put(3,14){\circle*{2}}
\put(-0.65,7){$\wedge$}
\end{picture}}

\newcommand{\pdtroisun}[3]{\begin{picture}(20,12)(-5,-1)
\put(3,7){\circle*{2}}
\put(-0.65,0){$\wedge$}
\put(6,0){\circle*{2}}
\put(0,0){\circle*{2}}
\put(5,5){\tiny #1}
\put(-5,-2){\tiny #2}
\put(9,-2){\tiny #3}
\end{picture}}

\newcommand{\pdquatresept}[4]{\begin{picture}(20,12)(-5,-1)
\put(0,0){\circle*{2}}
\put(7,0){\circle*{2}}
\put(0,7){\circle*{2}}
\put(7,7){\circle*{2}}
\put(0,0){\line(0,1){7}}
\put(7,0){\line(0,1){7}}
\put(.5,1.5){$\scriptstyle \diagup$}
\put(0,1.5){$\scriptstyle \diagdown$}
\put(-5,-2){\tiny #1}
\put(9,-2){\tiny #2}
\put(-5,5){\tiny #3}
\put(9,5){\tiny #4}
\end{picture}}

\input{xy}
\xyoption{all}

\newcommand{\DP}{\mathcal{DP}}
\newcommand{\WNP}{\mathcal{WNP}}
\newcommand{\PP}{\mathcal{PP}}
\newcommand{\PF}{\mathcal{PF}}
\newcommand{\D}{\mathcal{D}}
\newcommand{\F}{\mathcal{F}}
\renewcommand{\P}{\mathcal{P}}
\newcommand{\R}{\mathcal{R}}
\newcommand{\h}{\mathcal{H}}
\renewcommand{\S}{\mathfrak{S}}
\newcommand{\prodg}{\rightsquigarrow}
\newcommand{\prodh}{\lightning}
\newcommand{\tdelta}{\tilde{\Delta}}
\newcommand{\rond}[1]{*++[o][F-]{#1}}

\title{Algebraic structures on double and plane posets}
\date{}
\author{Loïc Foissy \\ \\
{\small{\it Laboratoire de Mathématiques, Université de Reims}}\\
\small{{\it Moulin de la Housse - BP 1039 - 51687 REIMS Cedex 2, France}}\\
\small{e-mail: loic.foissy@univ-reims.fr}}

\newtheorem{defi}{\indent Definition}
\newtheorem{lemma}[defi]{\indent Lemma}
\newtheorem{cor}[defi]{\indent Corollary}
\newtheorem{theo}[defi]{\indent Theorem}
\newtheorem{prop}[defi]{\indent Proposition}

\newenvironment{proof}{{\bf Proof.}}{\hfill $\Box$}

\begin{document}

\maketitle

ABSTRACT. We study the Hopf algebra of double posets and two of its Hopf subalgebras, the Hopf algebras of plane posets and of posets "without N". 
We prove that they are free, cofree, self-dual, and we give an explicit Hopf pairing on these Hopf algebras. We also prove that they are free $2$-$As$ algebras;
in particular, the Hopf algebra of posets "without N" is the free $2$-$As$ algebra on one generator.
We deduce a description of the operads of $2$-$As$ algebras and of $B_\infty$ algebras in terms of plane posets.\\


AMS CLASSIFICATION. 16W30, 06A11.

\tableofcontents

\section*{Introduction}

The Hopf algebra of double posets is introduced by Malvenuto and Reutenauer in \cite{Reutenauer}:
a \emph{ double poset} is a finite set with two partial orders (definition \ref{1}); the vector space $\h_{\DP}$ generated by the set $\DP$ inherits two products,
here denoted by $\prodg$ and $\prodh$ (definition \ref{2}), and a coproduct $\Delta$ given by the ideals of the posets
(proposition \ref{28}), such that $(\h_{\DP},\prodg,\Delta)$ is a graded, connected Hopf algebra. Moreover, a Hopf pairing $\langle-,-\rangle$ 
is combinatorially defined on $\h_{\DP}$ (definition \ref{30}).

We here study this Hopf algebra $\h_{\DP}$ and some of its Hopf subalgebras:
the Hopf algebra of plane posets $\h_{\PP}$ (definition \ref{10}), the algebra of WN posets $\h_{\WNP}$ (definition  \ref{21})
and the algebra of plane forests $\h_{\PF}$. We shall say that a double poset $P$ is \emph{ plane} if its two partial orders $\leq_h$ and $\leq_r$ satisfy
a certain compatibility condition . We shall say that a plane poset is \emph{ WN} ("without N") if it does not contain $\pquatresix$ 
nor $\pquatrecinq$ as plane subposets. Finally, \emph{ plane forests} are plane posets whose Hasse graph is a rooted forest.

Note that $\h_{\PF}$ is equal to the non commutative Connes-Kreimer Hopf algebra of plane forests, 
introduced in \cite{Foissy,Holtkamp}.  Using the involution $\iota$ permuting the two partial orders of any double poset, we prove
that the restriction of the pairing $\langle-,-\rangle$ to any of these subalgebras is non-degenerate, with the possible exception of $\h_{\DP}$
if the base ring does not contain $\mathbb{Q}$. \\

The notion of $2$-$As$ algebra is introduced and studied in \cite{Loday,Loday2}: a $2$-$As$ algebra is an algebra with two associative products,
sharing the same unit. We prove here that $\h_{\DP}$, $\h_{\PP}$ and $\h_{\WNP}$, with their products $\prodg$ and $\prodh$,
are free $2$-$As$ algebras. In particular, the last one is the free $2$-$As$ algebra on one generator $\tun$: this gives an alternative
description of free $2$-$As$ algebras. As a consequence, the space of primitive elements of these Hopf algebras inherit a structure of 
free $B_\infty$-algebras.  Recall that a $B_\infty$-algebra is a vector space $V$ with a family of linear maps $[-,-]_{m,n}:A^{\otimes m}\otimes A^{\otimes n}
\longrightarrow A$ for all $m,n\geq 1$; if we consider the unique coalgebra morphism $\star_V:T(V)\otimes T(V) \longrightarrow T(V)$,
such that for all $m,n \in \mathbb{N}^*$, for all $x_1,\cdots,x_m,y_1,\cdots,y_n \in V$:
$$\pi_V((x_1\otimes \cdots \otimes x_m) \star_V (y_1\otimes \cdots \otimes y_n))=[ x_1,\cdots,x_m;y_1,\cdots,y_n]_V,$$
where $\pi_V$ is the canonical projection on $V$, then $(T(V),\star_V,\Delta)$ is a Hopf algebra.
Here, $T(V)$ is given its deconcatenation coproduct $\Delta$ (see \cite{Loday2} for more details and references about $B_\infty$ algebras). 
Using the dual product of the coproduct $\h_{\WNP}$,  we deduce a combinatorial description of the operad of $B_\infty$-algebras, in terms of double posets.\\

This text is organised as follows: the first section introduces the algebra of double posets. It is shown that $(\h_{\DP},\prodg)$
and $(\h_{\DP},\prodh)$ are two free algebras, generated respectively by the set of $1$- and $2$-indecomposable double posets (definition \ref{5}).
We also prove that $\h_{\DP}$ is, as a $2$-$As$ algebra, by the set of double posets that are both $1$- and $2$-indecomposable.

We introduce plane and WN posets, as well as the corresponding Hopf algebras, in the second section. We show that the condition 
for a plane poset $P=(P,\leq_h,\leq_r)$ to be $1$-indecomposable can be reformulated in terms of connectivity of the Hasse diagram of $(P,\leq_h)$, 
a result that may be false in general for double posets (proposition \ref{20}). We prove that $\h_{\PP}$ and $\h_{\WNP}$ are free $2$-$As$ algebras, 
the last one being generated by a single element.

The coproduct of $\h_{\DP}$ is introduced in the third section. It is also proved that $\h_{\DP}$, $\h_{\PP}$ and $\h_{\WNP}$
are $2$-$As$ bialgebras, in the sense of \cite{Loday}. They are all free and cofree.

The fourth section deals with the pairing. We prove that its restrictions to $\h_{\DP}$, $\h_{\PP}$ and $\h_{\WNP}$ are non-degenerate,
using a total order on the sets of double posets and the involution $\iota$. 

The last section is dedicated to a combinatorial description of the operad of $B_\infty$ algebras, with the help of indexed WN posets.
We first give an alternative description of the free $2$-$As$ algebra on one generator, and deduce a description of the free $B_\infty$ algebras
in terms of $1$-indecomposable decorated WN posets. The description of the operads $B_\infty$ and $2$-$As$ is a consequence of these results.\\

The author thanks Professor Christophe Reutenauer for his helpful comments and remarks.\\

{\bf Notations}.  \begin{enumerate}
\item In the whole text, $K$ is a commutative field. Any algebra, coalgebra, Hopf algebra\ldots of the text will be taken over $K$.
\item Let $H$ be a Hopf algebra. Its augmentation ideal is given a coassociative, non counitary coproduct $\tdelta$ defined by
$\tdelta(x)=\Delta(x)-x \otimes 1-1\otimes x$.
\end{enumerate}

\section{Double posets}

We refer to \cite{Stanley} for classical definitions and results on posets.

\subsection{Definitions}

\begin{defi}\label{1} \textnormal{
\cite{Reutenauer} A \emph{ double poset} is a triple $(P,\leq_1,\leq_2)$, where $P$ is a finite set and $\leq_1$, $\leq_2$ are two partial orders on $P$.
The set of isoclasses of double posets will be denoted by $\DP$.
The set of isoclasses of double posets of cardinality $n$ will be denoted by $\DP(n)$ for all $n\geq 0$.
}\end{defi}

\begin{defi}\label{2}\textnormal{
Let $P$ and $Q$ be two elements of $\DP$. 
\begin{enumerate}
\item We define $P\prodg Q\in \DP$ by:
\begin{itemize}
\item $P\prodg Q$ is the disjoint union of $P$ and $Q$ as a set.
\item $P$ and $Q$ are double subposets of $P \prodg Q$.
\item For all $x\in P$, $y\in Q$, $x \leq_2 y$ in $P \prodg Q$ and $x$ and $y$ are not comparable for $\leq_1$ in $P \prodg Q$.
\end{itemize}
\item We define $P\prodh Q\in \DP$ by:
\begin{itemize}
\item $P\prodh Q$ is the disjoint union of $P$ and $Q$ as a set.
\item $P$ and $Q$ are double subposets of $P \prodh Q$.
\item For all $x\in P$, $y\in Q$, $x \leq_1 y$ in $P \prodh Q$ and $x$ and $y$ are not comparable for $\leq_2$ in $P \prodh Q$.
\end{itemize}
\end{enumerate}} \end{defi}

{\bf Remark.} The product $\prodg$ is called \emph{ composition} in \cite{Reutenauer}.

\begin{prop}
The products $\prodg$ and $\prodh$ are associative.
\end{prop}

\begin{proof} Let us take $P,Q,R \in \DP$. Then $(P\prodg Q)\prodg R$ and $P\prodg (Q\prodg R)$ are both equal to the double poset $S$ defined by:
\begin{itemize}
\item $S$ is the disjoint union of $P$, $Q$ and $R$ as a set.
\item $P$, $Q$ and $R$ are double subposets of $S$.
\item For all $x\in P$, $y\in Q$, $z \in R$, $x \leq_2 y \leq_2 z$ in $S$ and $x$, $y$ and $z$ are not comparable for $\leq_1$ in $S$.
\end{itemize}
So $\prodg$ is associative. The proof is similar for $\prodh$. \end{proof}

\begin{defi}\textnormal{
Let us denote by $\h_{\DP}$ the $K$-vector space generated by $\DP$. We extend $\prodg$ and $\prodh$ by linearity on $\h_{\DP}$.
As a consequence, $(\h_{\DP},\prodg,\prodh)$ is a $2$-$As$-algebra  \cite{Loday,Loday2}, that is to say an algebra with two associative products
sharing the same unit, the empty double poset $1$.
}\end{defi}

{\bf Remark.} We shall see that it is a free $2$-$As$-algebra in theorem \ref{9}.

\subsection{Indecomposable double posets}

\begin{defi}\label{5}\textnormal{Let $P$ be a double poset. 
\begin{enumerate}
\item We shall say that $P$ is \emph{ $1$-indecomposable} if for any $I \subseteq P$:
$$(\forall x \in I,\:\forall y\in P\setminus I,\:x \leq_2 y \mbox{ and }x, y\mbox{ are not }\leq_1\mbox{-comparable})
\Longleftrightarrow (I=\emptyset \mbox{ or }I=P).$$
\item We shall say that $P$ is \emph{ $2$-indecomposable} if for any $I \subseteq P$:
$$(\forall x \in I,\:\forall y\in P\setminus I,\:x \leq_1 y \mbox{ and }x, y\mbox{ are not }\leq_2\mbox{-comparable})
\Longleftrightarrow (I=\emptyset \mbox{ or }I=P).$$
\item We shall say that $P$ is \emph{ $1,2$-indecomposable} if it is both $1$- and $2$-indecomposable.
\end{enumerate}}\end{defi}

{\bf Remark.} In other words, $P$ is not $1$-indecomposable if there exists $\emptyset \subsetneq I,J \subsetneq P$, such that $P=I \prodg J$;
$P$ is not $2$-indecomposable if there exists $\emptyset \subsetneq I,J \subsetneq P$, such that $P=I \prodh J$.

\begin{prop} \label{6}
Let $P$ be a double poset.
\begin{enumerate}
\item $P$ can be uniquely written as $P=P_1 \prodg \ldots \prodg P_k$, where $P_1,\ldots,P_k$ are $1$-indecomposable double posets.
\item $P$ can be uniquely written as $P=P'_1 \prodh \ldots \prodh P'_l$, where $P'_1,\ldots,P'_l$ are $2$-indecomposable double posets.
\end{enumerate}\end{prop}

\begin{proof} We only prove the first point. The proof of the second point in similar, permuting $\leq_1$ and $\leq_2$.\\

\emph{ Existence.} By induction on $n=Card(P)$. If $n=1$, then $P$ is $1$-indecomposable, so we choose $k=1$ and $P_1=P$.
Let us assume the result at all rank $<n$. If $P$ is $1$-indecomposable, it can be written as $P=P$. If not, there exists
$\emptyset \subsetneq I,J \subsetneq P$, such that $P=I \prodg J$. Then the induction hypothesis holds for $I$ and $J$.
So $I=P_1\prodg \ldots \prodg P_s$ and $J=P_{s+1}\prodg\ldots \prodg P_k$, where the $P_i$ are $1$-indecomposable.
Hence, $P=I \prodg J=P_1\prodg\ldots \prodg P_k$.\\

\emph{ Unicity.} Let us assume that $P=P_1\prodg \ldots \prodg P_k=Q_1\prodg \ldots \prodg Q_l$, where the $P_i$ and the $Q_j$ are $1$-indecomposable.
The $P_i$'s and the $Q_j$'s are part of $P$; let us consider $I=P_1 \cap Q_1$.
For all $x \in I$, $y\in Q_1\setminus I=Q_1\cap (P_2 \prodg \ldots \prodg P_k)$, $x \leq_2 y$ and $x,y$ are not $\leq_1$-comparable.
As $Q_1$ is $1$-indecomposable, $I=Q_1$ or $I=\emptyset$. Let $x \in P$ be a minimal element for $\leq_2$.
There exists $1\leq i\leq k$, such that $x \in P_i$. If $i\geq 2$, then for any $y\in P_1$, $y<_2 x$: contradicts the minimality of $x$.
So $x\in P_1$ and, similarly, $x\in Q_1$. So $I\neq \emptyset$, so $I=Q_1$ and $Q_1 \subseteq P_1$. By symmetry, $P_1=Q_1$.
We then deduce that $P_2 \prodg \ldots \prodg P_k=Q_2 \prodg \ldots \prodg Q_l$. Using repeatedly the same arguments, we prove
that $k=l$, $P_2=Q_2,\ldots,P_k=Q_k$. \end{proof}\\

{\bf Remark.} As a consequence, $(\h_{\DP},\prodg)$ is freely generated by the set of $1$-indecomposable double posets
and $(\h_{\DP},\prodh)$ is freely generated by the set of $2$-indecomposable double posets.

\begin{lemma}
Let $P$ be a double poset.
\begin{enumerate}
\item If $P$ is not $1$-indecomposable, then $P$ is $2$-indecomposable.
\item If $P$ is not $2$-indecomposable, then $P$ is $1$-indecomposable.
\end{enumerate}\end{lemma}

\begin{proof} Note that the first point is the contraposition of the second point.
Let us assume that $P$ is not $2$-indecomposable. We can write $P=P'_1\prodh \ldots \prodh P'_l$, with $l\geq 2$, $P'_1,\ldots,P'_l$ $2$-indecomposable.
Let $\emptyset \subsetneq I \subseteq P$, such that for all $x \in I$, $\forall y\in P\setminus I$, $x \leq_2 y$ and $x,y$ are not $\leq_1$-comparable.

Let us choose $x\in I$. There exists $1\leq i\leq k$, such that $x\in P'_i$.
If $y\in P'_j$, with $j\neq i$, then $x\leq_1 y$ if $i<j$ or $x \geq_1 y$ if $i>j$, so $x,y$ are $\leq_1$-comparable. By hypothesis on $I$, $y\in I$.
So $P'_j \subseteq I$ if $j\neq i$. 

Let us now choose $j\neq i$ (this is possible, as $l\geq 2$) and $y\in P'_j$. Then $y\in I$ and if $z\in P'_i$, $y,z$ are $\leq_1$-comparable. 
So $z\in I$ and $P'_i \subseteq I$. As a consequence, $I=P$ and $P$ is $1$-indecomposable. \end{proof}\\

As an immediate consequence:

\begin{prop} \label{8}
Let $P$ be a double poset, not equal to $1$. One, and only one, of the following conditions holds:
\begin{itemize}
\item $P$ is $1,2$-indecomposable.
\item $P$ is $1$-indecomposable and not $2$-indecomposable.
\item $P$ is $2$-indecomposable and not $1$-indecomposable.
\end{itemize}\end{prop}

\subsection{The $2$-$As$ algebra $\h_{\DP}$}

\begin{theo} \label{9}
As a $2$-$As$ algebra, $\h_{\DP}$ is freely generated by the set of $1,2$-indecomposable double posets.
\end{theo}

\begin{proof} Let $(A,.,*)$ be a $2$-$As$ algebra and let $a_P\in A$ for all $1,2$-indecomposable double poset $P$.
We have to prove that there exists a unique morphism of $2$-$As$ algebras $\phi:\h \longrightarrow A$, such that
$\phi(P)=a_P$ for all $1,2$-indecomposable double poset $P$.

\emph{ Existence.} We define $\phi(P)$ for $P\in \DP(n)$ by induction on $n$ in the following way:
\begin{itemize}
\item $\phi(1)=1$.
\item If $P$ is $1,2$-indecomposable, $\phi(P)=a_P$.
\item If $P$ is $1$-indecomposable and not $2$-indecomposable, let us put $P=P'_1\prodh \cdots \prodh P'_l$, where the $P'_i$'s are $2$-indecomposable; 
then $\phi(P)=\phi(P'_1)*\cdots*\phi('P_l)$.
\item If $P$ is not $1$-indecomposable and $2$-indecomposable, let us put $P=P_1\prodg \cdots \prodg P_k$, where the $P_i$'s are $1$-indecomposable; 
then $\phi(P)=\phi(P_1).\cdots.\phi(P_k)$.
\end{itemize}
By propositions \ref{6} and \ref{8}, this perfectly defines $\phi$.\\

Let $P,Q \in \DP$. We put $P=P_1\prodg \cdots \prodg P_k$ and $Q=Q_1\prodg\cdots \prodg Q_l$,
where the $P_i$'s and the $Q_i$'s are $1$-indecomposable  double posets. Then:
$$P\prodg Q=P_1\prodg \cdots \prodg P_k\prodg Q_1\prodg\cdots \prodg Q_l,$$
so, by definition of $\phi$:
$$\phi(P\prodg Q)=\phi(P_1)\cdots\phi(P_k)\phi(Q_1)\cdots \phi(Q_l)=(\phi(P_1)\cdots\phi(P_k))(\phi(Q_1)\cdots \phi(Q_l))=\phi(P)\phi(Q).$$
Similarly, we can prove that $\phi(P\prodh Q)=\phi(P)*\phi(Q)$. So $\phi$ satisfies the required properties.\\

\emph{ Unicity.} Such a morphism has to satisfy all the conditions of the existence part, so is equal to $\phi$. \end{proof}

\section{Plane posets}

\subsection{Definition}

\begin{defi}\label{10}\textnormal{
A \emph{ plane poset} is a double poset $(P,\leq_h,\leq_r)$ such that for all $x, y\in P$, such that $x\neq y$,
$x$ and $y$ are comparable for $\leq_h$ if, and only if, $x$ and $y$ are not comparable for $\leq_r$.
The set of isoclasses of plane posets will be denoted by $\PP$. 
For all $n \in \mathbb{N}$, the set of isoclasses of plane posets of cardinality $n$ will be denoted by $\PP(n)$.
}\end{defi}

{\bf Remark.} Let $P \in \PP$ and let $x,y\in P$. Then $(x \leq_h y)$ or $(x \geq_h y)$ or $(x \leq_r y)$ or $(x \geq_r y)$.
Moreover, if $x \neq y$, then these four conditions are two-by-two incompatible.\\

We shall give a graphical representation of plane posets. If $(P,\leq_h,\leq_r)$ is a plane poset, we shall represent the Hasse graph of $(P,\leq_h)$ 
such that if $x <_r y$ in $P$, then $y$ is more on the right than $x$ in the graph. This justifies the notations $\leq_h$ ($h$ is for "high")
and $\leq_r$ ($r$ is for "right") instead of $\leq_1$ and $\leq_2$.\\

{\bf Examples.} \begin{enumerate}
\item Here are the plane posets of cardinal $\leq 4$:
\begin{eqnarray*}
\PP(0)&=&\{\emptyset\},\\
\PP(1)&=&\{\tun\},\\
\PP(2)&=&\{\tun\tun,\tdeux\},\\
\PP(3)&=&\{\tun\tun\tun,\tun\tdeux,\tdeux\tun,\ttroisun,\ttroisdeux,\ptroisun\},\\
\PP(4)&=&\left\{\begin{array}{c}
\tun\tun\tun\tun,\tun\tun\tdeux,\tun\tdeux\tun,\tdeux\tun\tun,\tun\ttroisun,\ttroisun\tun,
\tun\ttroisdeux,\ttroisdeux\tun,\tun\ptroisun,\ptroisun\tun,\tdeux\tdeux,\tquatreun,\\
\tquatredeux,\tquatretrois,\tquatrequatre,\tquatrecinq,\pquatreun,\pquatredeux,
\pquatretrois,\pquatrequatre,\pquatrecinq,\pquatresix,\pquatresept,\pquatrehuit
\end{array}\right\}.\end{eqnarray*}
We shall prove elsewhere \cite{Foissy4} that $Card(\PP(n))=n!$ for all $n\geq 0$. 
\item Let $F$ be a plane forest. We defined in \cite{Foissy2,Foissy3} two partial orders on $F$, which makes it a plane poset.
More precisely, the Hasse graph of $(F,\leq_h)$ is the graph $F$, the edges being oriented from the root to the leaves.
The partial order $\leq_r$ is defined by two vertices $x,y$ which are not comparable for $\leq_h$ in the following way: if $F=t_1\ldots t_n$,
with $x$ a vertex if $t_i$ and $y$ a vertex of $t_j$, 
\begin{itemize}
\item $x \leq_r y$ if $i<j$ and $x \geq_r y$ if $i>j$.
\item If $i=j$, then $x \leq_r y$ if $F$ if, and only if $x \leq_r y$ in the forest obtained by deleting the root of $t_i$.
\end{itemize}
As a conclusion, the Hasse graph of $(F,\leq_h,\leq_r)$ is the plane forest $F$ itself. Such a plane poset will be called a \emph{forest}.
The set of plane forests will be denoted by $\PF$; for all $n\geq 0$, the set of plane forests with $n$ vertices will be denoted by $\PF(n)$.
For example:
\begin{eqnarray*}
\PF(1)&=&\{\tun\},\\
\PF(2)&=&\{\tun\tun, \tdeux\},\\
\PF(3)&=&\{\tun\tun\tun,\tun\tdeux,\tdeux\tun,\ttroisun,\ttroisdeux\},\\
\PF(4)&=&\left\{\tun\tun\tun\tun,\tun\tun\tdeux,\tun\tdeux\tun,\tdeux\tun\tun,\tun\ttroisun,\ttroisun\tun,
\tun\ttroisdeux,\ttroisdeux\tun,\tdeux\tdeux,\tquatreun,\tquatredeux,\tquatretrois,\tquatrequatre,\tquatrecinq\right\}.
\end{eqnarray*}
\end{enumerate}

\begin{prop} \label{11}
Let $P\in\PP$. We define a relation $\leq$ on $P$ by:
$$(x\leq y) \mbox{ if } (x\leq_h y\mbox{ or } x \leq_r y).$$
Then $\leq$ is a total order on $P$.
\end{prop}

\begin{proof} For any $x\in P$, $x\leq x$ as $x\leq_h x$. Let us assume that $x\leq y$ and $y\leq z$. Then three cases are possible.
\begin{itemize}
\item ($x\leq_h y$ and $y\leq_h z$) or ($x\leq_r y$ and $y\leq_r z$). Then ($x\leq_h z$) or ($x \leq_r z$), so $x\leq z$.
\item $x\leq_h y$ and $y\leq_r z$. As $P$ is plane, then $x$ and $z$ are comparable for $\leq_h$ or $\leq_r$. 
If $x\leq_h z$ or $x\leq_r z$, then $x \leq z$. It remains two subcases.
\begin{itemize}
\item If $z\leq_r x$, then $y\leq_r z \leq_r x$, so $y \leq_r x$. Moreover, $x\leq_h y$, so, as $P$ is plane, $x=y$ and finally $x\leq z$.
\item If $z \leq_h x$, then $z \leq_h x \leq_h y$, so $z\leq_h y$. Moreover, $y\leq_r z$, so, as $P$ is plane, $y=z$ and finally $x \leq z$.
\end{itemize}
\item $x\leq_r y$ and $y\leq_h z$. Similar proof.
\end{itemize}
Let us assume that $x\leq y$ and $y\leq x$. Two cases are possible.
\begin{itemize}
\item ($x \leq_h y$ and $y\leq_h x$) or ($x \leq_h y$ and $y\leq_r x$). Then $x=y$.
\item ($x \leq_r y$ and $y\leq_h x$) or  ($x \leq_h y$ and $y\leq_r x$). As $P$ is plane, $x=y$.
\end{itemize}
So $\leq$ is an order on $P$. Moreover, by definition of a plane poset, if $x,y\in P$, then $x\leq y$ or $x\geq y$, so $\leq$ is total. \end{proof}\\

{\bf Notations.} \begin{enumerate}
\item Let $n \in \mathbb{N}$. We denote by $\wp_n$ the double poset with $n$ elements such that for all $x,y\in \wp_n$, the following assertions are equivalent:
\begin{enumerate}
\item $x$ and $y$ are comparable for $\leq_1$.
\item $x$ and $y$ are comparable for $\leq_2$.
\item $x=y$.
\end{enumerate}
\item $\tddeux{$1$}{$2$}$ is the double poset with two elements $x,y$ such that $x\leq_1 y$ and $x\leq_2 y$.
\item $\tddeux{$2$}{$1$}$ is the double poset with two elements $x,y$ such that $x\leq_1 y$ and $y\leq_2 x$.
\end{enumerate}

{\bf Remark.} Note that $\tddeux{$1$}{$2$}$ and $\tddeux{$2$}{$1$}$ are not plane posets;  $\wp_n$ is plane if, and only if, $n=0$ or $1$.

\begin{prop} 
Let $P$ be a double poset. Then $P$ is plane if, and only if, it does not contain any double subposet isomorphic to $\wp_2$, $\tddeux{$1$}{$2$}$
or $\tddeux{$2$}{$1$}$.
\end{prop}

\begin{proof} $\Longrightarrow$. Let $x,y\in P$, $x\neq y$. If $x,y$ are comparable for $\leq_1$, then $\{x,y\} \neq \wp_2$; moreover,
$x,y$ are not comparable for $\leq_2$ as $P$ is plane, so $\{x,y\}\neq \tddeux{$1$}{$2$}$ and $\tddeux{$2$}{$1$}$. 
If $x,y$ are not comparable for $\leq_1$, then $\{x,y\}\neq \tddeux{$1$}{$2$}$ and $\tddeux{$2$}{$1$}$; moreover, $x,y$ are comparable for $\leq_2$,
so $\{x,y\} \neq \wp_2$.\\

$\Longleftarrow$. Let $x,y\in P$, $x\neq y$. As $\{x,y\} \neq \wp_2$, $x,y$ are comparable for $\leq_1$ or $\leq_2$.
As $\{x,y\} \neq \tddeux{$1$}{$2$}$ and $\tddeux{$2$}{$1$}$, they are not comparable for both of the partial order $\leq_1$ and $\leq_2$.
So $P$ is plane. \end{proof}

\subsection{Can every poset become a plane poset?}

We here give a family of counterexamples of posets $(X,\leq_h)$ such that there does not exist a partial order $\leq_r$ making $(X,\leq_h)$ a plane poset.

\begin{prop}
Let $N\geq 1$. The poset $X_N$ has $2N$ vertices $x_{\overline{1}},\ldots,x_{\overline{N}}$ and $y_{\overline{1}},\ldots,y_{\overline{N}}$
indexed by $\mathbb{Z}/N\mathbb{Z}$. Its partial order is given by $x_{\overline{i}}\leq_h y_{\overline{i}}$ and $x_{\overline{i}}\leq_h y_{\overline{i+1}}$
for all $i\in \mathbb{Z}/N\mathbb{Z}$. If $N \geq 3$, there is no plane poset of the form $(X_N,\leq_h,\leq_r)$.
\end{prop}

Here are the Hasse graphs of $X_3$ and $X_4$:
$$
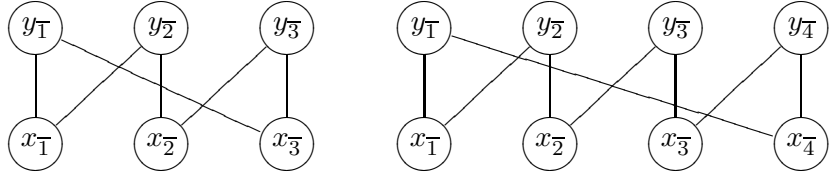
\xymatrix{\rond{y_{\overline{1}}}&\rond{y_{\overline{2}}}&\rond{y_{\overline{3}}}\\
\rond{x_{\overline{1}}}\ar@{-}[u] \ar@{-}[ru]&\rond{x_{\overline{2}}}\ar@{-}[u] \ar@{-}[ru]&\rond{x_{\overline{3}}}\ar@{-}[u] \ar@{-}[llu]}\hspace{1cm}
\xymatrix{\rond{y_{\overline{1}}}&\rond{y_{\overline{2}}}&\rond{y_{\overline{3}}}&\rond{y_{\overline{4}}}\\
\rond{x_{\overline{1}}}\ar@{-}[u] \ar@{-}[ru]&\rond{x_{\overline{2}}}\ar@{-}[u] \ar@{-}[ru]
&\rond{x_{\overline{3}}}\ar@{-}[u] \ar@{-}[ru]&\rond{x_{\overline{4}}}\ar@{-}[u] \ar@{-}[lllu]}$$

\begin{proof} Let us assume that there exists a plane poset $(X,\leq_h,\leq_r)$. As $x_{\overline{1}}$ and $x_{\overline{2}}$ are not comparable
for $\leq_h$, they are comparable for $\leq_r$. Let us assume for example that $x_{\overline{1}}\leq_r x_{\overline{2}}$ (the proof would be similar if
$x_{\overline{1}}\geq_r x_{\overline{2}}$). Let us prove by induction on $i$ that $x_{\overline{i}}\leq_r x_{\overline{i+1}}$. This is immediate for $i=1$. 
Let us assume that $x_{\overline{i}}\leq_r x_{\overline{i+1}}$. Then $x_{\overline{i}}\leq_r x_{\overline{i+1}}\leq_h y_{\overline{i+2}}$, 
so $x_{\overline{i}}\leq y_{\overline{i+2}}$. As $N\geq 3$, $x_{\overline{i}}$ and $y_{\overline{i+2}}$ are not comparable for $\leq_h$, 
so $x_{\overline{i}}\leq_r y_{\overline{i+2}}$. If $y_{\overline{i+1}} \geq_r y_{\overline{i+2}}$, then $x_{\overline{i}}\leq_r y_{\overline{i+2}} \leq_r y_{\overline{i+1}}$,
so $x_{\overline{i}}\leq_r y_{\overline{i+1}}$ and $x_{\overline{i}}\leq_h y_{\overline{i+1}}$: contradiction. So $y_{\overline{i+1}} \leq_r y_{\overline{i+2}}$.
If $y_{\overline{i+1}} \geq_r x_{\overline{i+2}}$, then $x_{\overline{i+2}}\leq_r y_{\overline{i+1}} \leq_r y_{\overline{i+2}}$,
so $x_{\overline{i+2}}\leq_r y_{\overline{i+2}}$ and $x_{\overline{i+2}}\leq_h y_{\overline{i+2}}$: contradiction. So $y_{\overline{i+1}} \leq_r x_{\overline{i+2}}$.
Finally, $x_{\overline{i+1}} \leq_h y_{\overline{i+1}} \leq_r x_{\overline{i+2}}$, so $x_{\overline{i+1}} \leq x_{\overline{i+2}}$. As they are not comparable for $\leq_h$,
$x_{\overline{i+1}} \leq_r x_{\overline{i+2}}$. We obtain $x_{\overline{1}} \leq_r \cdots \leq_r x_{\overline{N}} \leq_r x_{\overline{1}}$, 
so $x_{\overline{1}} = \cdots= x_{\overline{N}}$: absurd. \end{proof}\\

{\bf Remark.} Note that $X_1=\tdeux$ and $X_2=\pquatresept$ are plane posets.

\subsection{Products on plane posets}

Let $P,Q$ be two plane posets. It is not difficult to see that $P \prodg Q$ and $P\prodh Q$ are also plane posets.
Moreover, if $P$ is a plane poset, for any $I \subseteq P$, the double poset $I$ is also plane. As a consequence:

\begin{prop} 
Let $P$ be a double poset.
\begin{enumerate}
\item We write $P=P_1\prodg \ldots \prodg P_k$, where the $P_i$ are $1$-indecomposable. Then $P$ is plane if, and only if,
$P_1,\ldots,P_k$ are plane.
\item We write $P=P'_1\prodh \ldots \prodh P'_l$, where the $P'_j$ are $2$-indecomposable. Then $P$ is plane if, and only if,
$P'_1,\ldots,P'_l$ are plane.
\end{enumerate} \end{prop}

We denote by $\h_{\PP}$ the subspace of $\h_{\DP}$ generated by plane posets. It is a sub-$2$-$As$ algebra of $\h_{\DP}$.
The following result is proved as theorem \ref{9}:

\begin{theo} 
As a $2$-$As$ algebra, $\h_{\PP}$ is freely generated by the set of $1,2$-indecomposable plane posets.
\end{theo}

\subsection{Another description of indecomposable plane posets}

\begin{defi}\textnormal{
Let $P=(P,\preceq)$ be a poset. 
\begin{enumerate}
\item We define a relation $\R_P$ on $P$ in the following way: for all $x,y \in P$, $x \R_P y$ if there exists 
$x=x_0,x_1,\cdots,x_n=y$ elements of $P$, such that $x_i$ and $x_{i+1}$ are comparable for $\preceq$ for all $i\in \{0,\cdots,n-1\}$. 
This relation is clearly an equivalence.
\item The equivalence classes for $\R_P$ of $P$ will be called \emph{ connected components} of $P$.
If $P$ has only one connected component, it will be said \emph{ connected}. By convention, $\emptyset$ will not be considered as connected.
\end{enumerate}}\end{defi}

{\bf Remark.} The connected components of $P$ are the connected components of the Hasse graph of $(P,\preceq)$. \\

In the case of a double poset $P=(P,\leq_h,\leq_r)$, we can consider the two posets $(P,\leq_h)$ and $(P,\leq_r)$.

\begin{defi}\textnormal{Let $P=(P,\leq_h,\leq_r)$ be a double poset.
\begin{enumerate}
\item The connected components of $(P,\leq_h)$ will be called \emph{ $h$-connected components} of $P$.
If $P$ has only one $h$-connected component, we shall say that $P$ is \emph{ $h$-connected}.
\item The connected components of $(P,\leq_r)$ will be called \emph{ $r$-connected components} of $P$.
If $P$ has only one $r$-connected component, we shall say that $P$ is \emph{ $r$-connected}.
\item We shall say that $P$ is \emph{ biconnected} if it both $h$- and $r$-connected.
\end{enumerate}}\end{defi}

For example, $\tun$, $\pquatrecinq$ and $\pquatresix$ are biconnected. These are the only biconnected plane posets of degree $\leq 4$.

\begin{lemma} \label{19}
Let $P \in \PP$, and let $P_1,\cdots,P_k$ its $h$-connected components. For all $i\in \{1,\cdots,k\}$, let us fix an element $x_i\in P_i$.
If $i\neq j$, $x_i$ and $x_j$ are not in the same $h$-connected component of $P$, so are not comparable for $\leq_h$,
so are comparable for $\leq_r$. We suppose that the $P_i$'s are indexed such that $x_1\leq_r \cdots \leq_r x_k$.
Then $P=P_1\prodg \cdots \prodg P_k$.
\end{lemma}

\begin{proof} We have to show that if $1\leq i<j \leq k$, if $y_i \in P_i$ and $y_j \in P_j$, then $y_i \leq_r y_j$ 
and $y_i$ and $y_j$ are not comparable for $\leq_h$. As $P$ is a plane poset, the first assertion implies the second one.
As $y_i \R_h x_i$ and $y_j \R_h x_j$ there exists elements of $P$ such that:
\begin{itemize}
\item $s_0=x_i,\cdots, s_p=y_i$, $s_l$ and $s_{l+1}$ are comparable for all $l\in \{0,\cdots,p-1\}$.
\item $t_0=x_j,\cdots, t_q=y_j$, $t_l$ and $t_{l+1}$ are comparable for all $l\in \{0,\cdots,q-1\}$.
\end{itemize}
Note that all the $s_l$'s belong to $P_i$ and all the $t_l$'s belong to $P_j$, by definition of the relation $\R_h$. We can suppose that the $s_l$'s 
and the $t_l$'s are all distinct. Let us first prove that $s_l\leq_r t_0$ by induction on $l$. For $l=0$, this is the hypothesis of the lemma.
Let us suppose that $s_{l-1} \leq_r t_0$. As $s_l$ and $t_0$ are not in the same $h$-connected component of $P$, they are not comparable for $\leq_h$, 
so they are comparable for $\leq_r$. Let us suppose that $s_l \geq_r  t_0$. Then $s_l \geq_r t_0  \geq_r s_{l-1}$, so $s_l$ and $s_{l-1}$ are comparable 
for $\leq_r$: contradiction, they are distinct elements of $P$ and are comparable for $\leq_h$ in the plane poset $P$. So $s_l \leq_r t_0$.
As a conclusion, $y_i \leq_r t_0$. Similarly, an induction proves that $y_i \leq_r t_l$ for all $l$, so $y_i\leq_r y_j$. \end{proof}

\begin{prop} \label{20} 
Let $P$ be a double poset.
\begin{enumerate}
\item \begin{enumerate}
\item If $P$ is $h$-connected, then it is $1$-irreducible.
\item If $P$ is plane and $1$-irreducible, then it is $h$-connected.
\end{enumerate}
\item \begin{enumerate}
\item If $P$ is $r$-connected, then it is $2$-irreducible.
\item If $P$ is plane and $2$-irreducible, then it is $r$-connected.
\end{enumerate}
\end{enumerate}\end{prop}

\begin{proof} We only prove the first point. The second point is proved similarly, permuting the two partial orders of $P$.\\

1. $(a)$ Let us assume that $P$ is $h$-connected and not $1$-irreducible. There exists $\emptyset \subsetneq Q,R \subsetneq P$,
such that $P=Q\prodg R$. Let us choose $x \in Q$  and $y\in R$. As $P$ is $h$-connected, there exists 
$x_1,\ldots,x_k \in P$, such that $x_1=x$, $x_k=y$, and $x_i,x_{i+1}$ are $\leq_h$-comparable for all $1\leq i\leq k-1$.
As $x_1 \in Q$ and $x_2,x_1$ are $\leq_h$-comparable, necessarily $x_2\in Q$. Repeating this argument, we show that $x_3,\ldots,x_k\in Q$,
so $y\in Q$; contradiction, $Q$ and $R$ are disjoint. \\

1. $(b)$ Let us assume that $P$ is not $h$-connected. By lemma \ref{19}, we can write $P=P_1\prodg \cdots \prodg P_k$ with $k\geq 2$,
so $P$ is not $1$-irreducible. \end{proof}\\

{\bf Remark}. So a plane poset is $1$-irreducible if, and only if, it is $h$-connected. This result is false for double posets that are not plane.
For example, $\tddeux{$1$}{$3$} \tdun{$2$}$ is $1$-irreducible but not $h$-connected.
We used here the double poset $\tddeux{$1$}{$3$} \tdun{$2$}$, which has three elements $x,y,z$ such that:
\begin{itemize}
\item $x\leq_2 y\leq_2 z$.
\item $x \leq_1 z$, $x,y$ and $y,z$ are not comparable for $\leq_1$.
\end{itemize}





\subsection{WN posets}

We define a subset of $\PP$ in the following way:

\begin{defi}\label{21}\textnormal{
Let $P$ be a double poset. We shall say that $P$ is \emph{ WN} ("without N") if it is plane and does not have any subposet isomorphic to $\pquatrecinq$ nor $\pquatresix$.
The set of isoclasses of WN posets will be denoted by $\WNP$. 
For all $n \in \mathbb{N}$, the set of isoclasses of WN posets of cardinality $n$ will be denoted by $\WNP(n)$.
}\end{defi}

\begin{lemma}
\begin{enumerate}
\item Let $P\in \DP$. The following conditions are equivalent.
\begin{enumerate}
\item $P$ is WN.
\item The $h$-connected components of $P$ are WN.
\item The $r$-connected components of $P$ are WN.
\end{enumerate}
\item Let $P_1,P_2\in \DP$. The following conditions are equivalent:
\begin{enumerate}
\item $P_1$ and $P_2$ are WN.
\item $P_1\prodg P_2$ is WN.
\item $P_1\prodh P_2$ is WN.
\end{enumerate}
\end{enumerate} \end{lemma}

\begin{proof} The first point comes from the fact that $\pquatrecinq$ and $\pquatresix$ are $h$-connected and $r$-connected.
So $P$ contains  $\pquatrecinq$ or $\pquatresix$ if, and only if, one of its $h$-connected components contains $\pquatrecinq$ or $\pquatresix$,
if, and only if, one of its $r$-connected components contains $\pquatrecinq$ or $\pquatresix$.
The second point comes from the fact that the $h$-connected components of $P_1\prodg P_2$ are the $h$-connected components of $P_1$ and $P_2$
and the $r$-connected components of $P_1\prodh P_2$ are the $r$-connected components of $P_1$ and $P_2$. \end{proof}\\

{\bf Remark.} As a consequence, the subspace $\h_{\WNP}$ of $\h_{\PP}$ generated by $\WNP$ is a $2$-$As$ subalgebra.

\begin{prop} \label{23} \begin{enumerate}
\item Let $P\in \PP$. Then $P$ is $h$-connected or $P$ is $r$-connected.
\item Let $P\in \WNP$. If $P$ is biconnected, then $P=\tun$.
\end{enumerate}
\end{prop}

\begin{proof} 1. By proposition \ref{8}, $P$ is $1$-indecomposable or $2$-indecomposable, so
is $h$-connected or $r$-connected.\\
 
 2. Let $P$ be a WN double poset, of cardinal $n \geq 2$, $h$-connected and $r$-connected.
We choose $P$ such that $n$ is minimal. A direct consideration of double posets of cardinal $2$ and $3$ proves that $n\geq 4$.
Up to an isomorphism, we suppose that $P=\{1,\cdots,n\}$ as a totally ordered set.
We consider $Q=P-\{n\}$. by minimality of $n$, $Q$ is not $h$-connected or not $r$-connected.
For example, let us assume that $Q$ is not $h$-connected (the proof is similar in the other case, permuting $\leq_h$ and $\leq_r$).
We denote by $Q_1,\cdots,Q_k$ its $h$-connected components, such that $Q=Q_1\prodg\cdots \prodg Q_k$. Then $k\geq 2$.
As $P$ is $h$-connected, for all $i\in \{1,\cdots,k\}$, there exists $x_i \in Q_i$, such that $x_i \leq_h n$.
Moreover, $P$ is $r$-connected, so there exists $x\in Q$, $x\leq_r n$. Two cases are possible.
\begin{itemize}
\item If $x \in Q_1\cup \cdots \cup Q_{k-1}$, up to a change of $x$, as $P$ is $h$-connected, there exists $y \in Q_1\cup \cdots \cup Q_{k-1}$, such that
$y\leq_h x$ and $y\leq_h n$. Then the double subposet of $P$ formed by $x$, $y$, $x_k$ and $n$ is isomorphic to $\pquatrecinq$.
So $P$ is not WN: contradiction.
\item If $x\in Q_k$, up to a change of $x$, we can suppose that there exists $y\in Q_k$, such that 
$y\leq_h x$ and $y\leq_h n$. Then $x_1\leq_r x \leq_r n$, so $x_1\leq_r n$ and $x_1\leq_h n$: impossible, as $P$ is a double poset.
\end{itemize}
In both cases, this is a contradiction, so a WN double poset which is both $h$- and $r$-connected is equal to $\tun$. \end{proof}\\

Hence, propositions \ref{8} and \ref{20}  give:

\begin{prop} \label{24}
Let $P$ be a WN poset, not equal to $1$. One, and only one, of the following conditions holds:
\begin{itemize}
\item $P$ is equal to $\tun$.
\item $P$ is $1$-indecomposable and not $2$-indecomposable. Equivalently, $P$ is $h$-connected and not $r$-connected.
\item $P$ is $2$-indecomposable and not $1$-indecomposable. Equivalently, $P$ is $r$-connected and not $h$-connected.
\end{itemize}\end{prop}

We prove in the same way as theorem \ref{9} the following result:

\begin{theo} \label{25}
As a $2$-$As$ algebra, $\h_{\WNP}$ is freely generated by $\tun$.
\end{theo}


{\bf Notations}. We denote by $\WNP_h$ the set of $h$-connected WN posets and by $\WNP_r$ the set of $r$-connected WN posets. 
These sets are graded by the order. \\

Theorem \ref{25} implies that $\h_{\WNP}$ is isomorphic, as a Hopf algebra, to the Loday-Ronco $2$-$As$ free algebra on one generator. 
As a consequence, we obtain the following result:

\begin{prop}
We consider the formal series:
$$ \left\{ \begin{array}{rcl}
R_{\WNP}(x)&=&\displaystyle \sum_{n=0}^\infty card(\WNP(n)) x^n,\\
P_{\WNP_h}(x)&=&\displaystyle \sum_{n=1}^\infty card(\WNP_h(n)) x^n,\\
P_{\WNP_r}(x)&=&\displaystyle \sum_{n=1}^\infty card(\WNP_r(n)) x^n.
\end{array}\right.$$
Then:
$$P_{\WNP_h}(x)=P_{\WNP_r}(x)=\frac{1+x-\sqrt{1-6x+x^2}}{4},\hspace{1cm}
R_{\WNP}(x)=\frac{3-x-\sqrt{1-6x+x^2}}{2}.$$
In particular, $card(\WNP_h(n))$ is the $n$-th hyper-Catalan number.
\end{prop}

%
%
For example:
$$\begin{array}{|c|c|c|c|c|c|c|c|c|c|c|c|}
\hline n&0&1&2&3&4&5&6&7&8&9&10\\
\hline |\WNP(n)|&1&1&2&6&22&90&394&1\:806&8\:558&41\:586&206\:098\\
\hline |\WNP_h(n)|&0&1&1&3&11&45&197&903&4\:279&20\:793&103\:049\\
\hline \end{array}$$
The second row of this array is (up to the signs) sequence A086456 of \cite{Sloane}. 
The third row is sequence A001003 (little Schroeder numbers). Moreover, if $n \geq 2$, then $card(\WNP_h(n))=card(\WNP(n))/2$. \\
%
%

Plane forests are examples of WN forests, and more precisely:

\begin{lemma}
Let $P$ be a plane poset. Then  $P$ is a plane forest if, and only if, it does not contain $\ptroisun$.
\end{lemma}

\begin{proof} $\Longrightarrow$. Obvious.\\

$\Longleftarrow$. As $P$ does not contain $\ptroisun$, it does not contain $\pquatrecinq$ nor $\pquatresix$, so is WN.
We proceed by induction on $n=|P|$. If $n=1$, then $n=\tun$ is a plane forest. Let us assume that all double posets
that do not contain $\ptroisun$ of cardinality $<n$ are plane forests ($n\geq 2$). 
As $n \geq 2$, two cases can hold by proposition \ref{24}: 
\begin{itemize}
\item $P$ is not $h$-connected. We can write $P=P_1\prodg\ldots \prodg P_k$, with $k \geq 2$. By the induction hypothesis,
$P_1,\ldots,P_k$ are plane forests, so $P$ is also a plane forest.
\item $P$ is not $r$-connected. We can write $P=P_1\prodh\ldots \prodh P_l$, with $l \geq 2$, $P_1,\ldots, P_l$ $r$-connected.
By the induction hypothesis, $P_l$ is a plane forest. 
Let us take $1\leq i \leq l-1$. Let $x,y \in P_i$, not comparable for $\geq_h$. We can assume that $x \leq_r y$ without loss of generality.
Let us choose any $z \in P_l$. Then $x,y\leq_h z$, so the subposet of $P$ formed by $x,y$ and $z$ is equal to $\ptroisun$: contradiction.
Hence, $P_i$ is totally ordered by $\geq_h$, so is equal to $\tun^{\prodh n_i}$ for a particular $n_i$.
As $P_i$ is $r$-connected, $n_i=1$. As a conclusion, $P=\tun\prodh\ldots \prodh \tun \prodh P_l$, so $P$ is a plane tree. 
\end{itemize} 
In both cases, $P$ is a plane forest. \end{proof}

\subsection{Can a poset become a WN poset?}

\begin{prop}
Let $P=(P,\leq_h)$ be a finite poset. There exists a partial order $\leq_r$ such that $\tilde{P}=(P,\leq_h,\leq_r)$ is a WN poset if, and only if,
$P$ does not contain any subposet isomorphic to $\pquatresix$.
\end{prop}

\begin{proof} $\Longrightarrow$. Let us assume that there exists such a $\tilde{P}$, and that $P$ contains a subposet $Q$ equal to $\pquatresix$. 
Then, restricting $\leq_r$, there exists a partial order $\leq_r$ on $Q$ making $Q$ a plane poset $\tilde{Q}$.
It is easy to see that there are only two possibilities for $\tilde{Q}$: $\pquatresix$ or $\pquatrecinq$. As $\tilde{P}$ contains $\tilde{Q}$,
it is not WN: contradiction.\\

$\Longleftarrow$. By induction on $n=Card(P)$. It is obvious if $n=0,1$. Let us assume the result at all ranks $<n$.

\emph{ First case}. Let us assume that the Hasse graph of $P$ is not connected. We can write $P=P_1\sqcup \ldots \sqcup P_k$, with $k\geq 2$,
where the $P_i$'s are the connected components of the Hasse graph of $P$. By the induction hypothesis, we can construct
$\tilde{P}_1,\ldots,\tilde{P}_k$. We then take $\tilde{P}=\tilde{P}_1\prodg \ldots \prodg \tilde{P}_k$.\\

\emph{ Second case.} We now assume that the Hasse graph of $P$ is connected. Let $M$ be the set of maximal elements of $P$. We put:
$$I=\{x\in P\mid \forall y\in M,\:x\leq_h y\}.$$
Let us first prove that $I$ is non empty. Let $x\in P$, such that the number of elements $y\in M$ with $x\leq_h y$ is maximal.
If $x\notin I$, there exists $z\in M$, such that $x$ and $z$ are not comparable for $\leq_h$, 
as it is not possible to have $z \leq_h x$ by maximality of $z$. Moreover, there exists $z'\in M$, such that $x \leq_h z'$ (so $z\neq z'$). 
As the Hasse graph of $P$ is connected, up to a change of $z,z'$, there exists $y$, such that $y\leq_h z,z'$ (so $y\neq x$).
As $z,z'\in M$, they are not comparable for $\leq_h$, so $y\neq z,z'$.
If $y\leq_h x$, then $y\leq_h z$ and $y \leq_h m$ for all $m\in M$ such that $x\leq_h m$: contradicts the choice of $x$.
If $x\leq_h y$, then $x \leq_h z$: contradiction. So $x$ and $y$ are not comparable for $\leq_h$ (so $x\neq z,z'$). 
Finally, the subposet $Q=\{x,y,z,z'\}$ of $P$ is isomorphic to $\pquatresix$: contradiction.\\

We obtain then two subcases:
\begin{itemize}
\item $I=P$. Let $z,z'\in M$. Then $z,z'\in I$, so $z\leq_h z'$, $z'\leq_h z$ and finally $z=z'$, so $M$ is reduced to a single element $z$.
Moreover, for all $x\in P$, $x\leq z$. The induction hypothesis holds on $Q=P-\{z\}$, and we take $\tilde{P}=\tilde{Q}\prodh \tun$.
\item $\emptyset \subsetneq I\subsetneq P$. Let us take $x\in I$ and $y\in P\setminus I$. Let us assume we don't have $x\leq_h y$.
If $y\leq_h x$, then, as $x\in I$, for all $z\in M$, $y\leq_h z$ and $y\in I$: contradiction. So $x$ and $y$ are not comparable for $\leq_h$ (and $x\neq y$).
As $y\notin I$, there exists $z\in M$, $y$ and $z$ are not comparable for $\leq_h$ (so $y\neq z$). There also exists $z'\in M$, $y\leq_h z'$ (so $z\neq z'$). 
As $x\in I$, $x\leq_h z,z'$. As $x$ and $y$ are not comparable for $\leq_h$, $x\neq z$. As $z$ and $z'$ are two elements of $M$, 
they are not comparable for $\leq_h$, so $x \neq z'$. As $x$ and $y$ are not comparable for $\leq_h$, $y\neq z'$. Finally, the suposet $Q=\{x,y,z,z'\}$ of $Q$
isomorphic to $\pquatresix$: contradiction. So $x\leq_h y$.

We proved that for all $x\in I$, for all $y\in P\setminus I$, $x\leq_h y$. The induction hypothesis holds for $I$ and $P\setminus I$; we take
$\tilde{P}=\tilde{I}\prodh \widetilde{P\setminus I}$.
\end{itemize}

In all cases, we proved the existence of a convenient $\tilde{P}$. \end{proof}

\section{Hopf algebra structure on $\h_{\DP}$}

\begin{defi}\textnormal{
\cite{Reutenauer}. Let $P=(P,\leq_1,\leq_2)$ be a double poset and let $I\subseteq P$. We shall say that $I$ is a \emph{ $1$-ideal} of $P$ if
for all $x\in I$, $y\in P$, $x\leq_1 y$ implies that $y\in I$. We shall say shortly ideal instead of $1$-ideal in the sequel.
}\end{defi}

\begin{prop}\label{28}
$\h_{\DP}$ is given a Hopf algebra structure with the product $\prodg$ and the following coproduct: for any double poset $P$,
$$\Delta(P)=\sum_{\mbox{\scriptsize $I$ ideal of $P$}} (P\setminus I)\otimes I.$$
This Hopf algebra is graded by the cardinality of the double posets. Moreover, $(\h_{\DP},\prodh,\Delta)$ is an infinitesimal Hopf algebra.
\end{prop}

\begin{proof} It is proved in \cite{Reutenauer} that $(\h_{\DP},\prodg,\Delta)$ is a Hopf algebra. We give here the proof again for the reader's convenience.
Let us first show that $\Delta$ is coassociative. Let $P\in \DP$.
If $I$ is an ideal of $P$ and $J$ is an ideal of $I$, then clearly $J$ is also an ideal of $P$. If $K$ is an ideal of $P\setminus I$, then clearly
$I\cup K$ is an ideal of $P$. As a consequence:
$$(Id \otimes \Delta)\circ \Delta(P)=(\Delta \otimes Id) \circ \Delta(P)=
\sum_{\substack{P=I_1\sqcup I_2 \sqcup I_3\\\mbox{\scriptsize $I_2$ and $I_2\sqcup I_3$  ideals of }P}} I_1\otimes I_2 \otimes I_3.$$

Let $P,Q\in \DP$ and let $I$ be an ideal of $P\prodg Q$. Then $I\cap P$ is an ideal of $P$ and $I\cap Q$ is an ideal of $Q$.
In the other sense, if $I$ is an ideal of $P$ and $J$ is an ideal of $Q$, then $I\prodg J$ is an ideal of $P\prodg Q$. So:
$$\Delta(P\prodg Q)=\sum_{\mbox{\scriptsize $I,J$ ideals of $P,Q$}} (P\setminus I)\prodg (Q\setminus J)\otimes I\prodg J=\Delta(P)\prodg \Delta(Q),$$
so $(\h_{\DP},\prodg,\Delta)$ is a graded Hopf algebra. \\

Let $P,Q\in \DP$, non empty, and let $I$ be an ideal of $P\prodh Q$. If $I\cap P$ is nonempty, then $Q \subseteq I$.
So there are five types of ideals of $P\prodh Q$: 
$I=\emptyset$, or $I=P\prodh Q$, or $I=Q$, or $I$ is a non trivial ideal of $Q$, or $I\cap P$ is a non trivial of $P$ and $Q \subseteq I$. Hence:
\begin{eqnarray*}
\Delta(P\prodh Q)&=&P\prodh Q\otimes 1+1\otimes P\prodh Q+P \otimes Q+(P \otimes 1)\prodh \tdelta(Q)+\tdelta(P) \prodh (1 \otimes Q)\\
&=&P\prodh Q\otimes 1+1\otimes P\prodh Q+P \otimes Q+(P \otimes 1)\prodh \Delta(Q)-P\prodh Q \otimes 1-P\otimes Q\\
&&+\Delta(P) \prodh (1 \otimes Q)-P\otimes Q-1\otimes P\prodh Q\\
&=&(P \otimes 1)\prodh \Delta(Q)+\Delta(P) \prodh (1 \otimes Q)-P\otimes Q,
\end{eqnarray*}
so $(\h_{\DP},\prodh,\Delta)$ is an infinitesimal Hopf algebra. \end{proof}\\

{\bf Examples.} 
\begin{eqnarray*}
\tdelta(\tdeux)&=&\tun\otimes \tun\\
\tdelta(\ttroisun)&=&2\tdeux \otimes \tun+\tun\otimes \tun\tun\\
\tdelta(\ttroisdeux)&=&\tun\otimes \tdeux+\tdeux \otimes \tun\\
\tdelta(\ptroisun)&=&\tun\tun \otimes \tun+2\tun\otimes \tdeux\\
\tdelta(\tquatreun)&=&\tun\otimes \tun\tun\tun+3\tdeux\otimes \tun\tun+3\ttroisun\otimes \tun\\
\tdelta(\tquatredeux)&=&\ttroisdeux\otimes \tun+\ttroisun\otimes \tun+\tdeux\otimes \tdeux+\tdeux \otimes \tun\tun+\tun \otimes \tdeux\tun\\
\tdelta(\tquatretrois)&=&\ttroisdeux\otimes \tun+\ttroisun\otimes \tun+\tdeux\otimes \tdeux+\tdeux \otimes \tun\tun+\tun \otimes \tun\tdeux\\
\tdelta(\tquatrequatre)&=&2\ttroisdeux \otimes \tun+\tun\otimes \ttroisun+\tdeux \otimes \tun\tun\\
\tdelta(\tquatrecinq)&=&\tun\otimes \ttroisdeux+\tdeux \otimes \tdeux+\ttroisdeux \otimes \tun\\
\tdelta(\pquatreun)&=&\tun\tun\tun\otimes\tun+3\tun\tun\otimes\tdeux+3\tun \otimes \ptroisun\\
\tdelta(\pquatredeux)&=&\tun\otimes\ttroisdeux + \tun\otimes\ptroisun+\tdeux\otimes \tdeux+\tun\tun \otimes \tdeux+\tdeux \tun \otimes \tun\\
\tdelta(\pquatretrois)&=&\tun\otimes\ttroisdeux +\tun\otimes\ptroisun+\tdeux\otimes \tdeux+\tun\tun\otimes \tdeux +\tun \tdeux \otimes \tun\\
\tdelta(\pquatrequatre)&=&2\tun \otimes \ttroisdeux +\ptroisun\otimes \tun+ \tun\tun \otimes \tdeux\\
\tdelta(\pquatrecinq)&=&\tdeux\tun \otimes \tun+\ptroisun \otimes \tun+\tdeux \otimes \tdeux+\tun\tun\otimes \tun\tun+\tun\otimes \tun\tdeux+\tun \otimes \ttroisun\\
\tdelta(\pquatresix)&=&\tun\tdeux \otimes \tun+\ptroisun \otimes \tun+\tdeux \otimes \tdeux+\tun\tun\otimes \tun\tun+\tun\otimes \tdeux \tun+\tun \otimes \ttroisun\\
\tdelta(\pquatresept)&=&2\ptroisun\otimes \tun+2\tun \otimes \ttroisun+\tun\tun\otimes \tun\tun\\
\tdelta(\pquatrehuit)&=&\ttroisun \otimes \tun+2\tdeux \otimes \tdeux+\tun\otimes \ptroisun
\end{eqnarray*}

{\bf Remarks.} \begin{enumerate}
\item If $\P$ is a plane poset, then all its subposets are plane. If $\P$ is WN, then all its subposets are WN.
As a consequence, $\h_{\PP}$ and $\h_{\WNP}$ are Hopf subalgebras of $\h_{\DP}$.
\item Similarly, $\h_{\PF}$ is a Hopf subalgebra of $\h_{\DP}$. It is the coopposite of the Connes-Kreimer Hopf algebra of plane trees, 
as defined in \cite{Foissy2,Holtkamp}.
\end{enumerate}

As $(\h_{\DP},\prodh,\Delta)$ is an infinitesimal Hopf algebra, the coalgebra $(\h_{\DP},\Delta)$ is cofree, see \cite{Loday}. 
Similarly, $\h_{\PP}$ and $\h_{\WNP}$ are cofree. From the results of \cite{Foissy2}:

\begin{cor} \label{29} \begin{enumerate}
\item The Hopf algebras $\h_{\DP}$, $\h_{\PP}$ and $\h_{\WNP}$ are free and cofree.
\item The Hopf algebras $\h_{\DP}$, $\h_{\PP}$ and $\h_{\WNP}$ are self-dual.
\item If the characteristic of the base field is zero, the Lie algebras $Prim(\h_{\DP})$, $Prim(\h_{\PP})$ and $Prim(\h_{\WNP})$ are free.
\end{enumerate}\end{cor}

\section{Hopf pairing of $\h_{\DP}$}

\begin{defi}\textnormal{\label{30}
Let $P,Q$ be two elements of $\DP$. We denote by $S(P,Q)$ the set of bijections $\sigma:P\longrightarrow Q$ such that, for all $i,j \in P$:
\begin{itemize}
\item ($i\leq_1 j$ in $P$) $\Longrightarrow$ ($\sigma(i) \leq_2 \sigma(j)$ in $Q$).
\item ($\sigma(i)\leq_1 \sigma(j)$ in $Q$) $\Longrightarrow$ ($i \leq_2 j$ in $P$).
\end{itemize}}\end{defi}

{\bf Remark.} The elements of $(P,Q)$ are called \emph{ pictures} in \cite{Reutenauer}. 

\begin{theo}
We define a pairing $\langle-,-\rangle:\h_{\DP}\otimes \h_{\DP} \longrightarrow K$ by:
$$\langle P,Q \rangle=Card(S(P,Q)),$$
for all $P,Q \in \DP$. Then $\langle-,-\rangle$ is a homogeneous symmetric Hopf pairing on the Hopf algebra $\h_{\DP}=(\h_{\DP},\prodg,\Delta)$.
\end{theo}

\begin{proof} See \cite{Reutenauer}. Let us consider the following map:
$$\Upsilon: \left\{ \begin{array}{rcl}
S(P_1\prodg P_2,Q)&\longrightarrow &\displaystyle \bigcup_{\mbox{\scriptsize $I$ ideal of $Q$}} S(P_1,Q\setminus I) \times S(P_2,I)\\
\sigma &\longrightarrow &(\sigma_{\mid P_1},\sigma_{\mid P_2})\in S(P_1,Q\setminus \sigma(P_2))\times S(P_2,\sigma(P_2)).
\end{array}\right.$$
The proof essentially consists in showing that $\Upsilon$ is a bijection. \end{proof}\\

{\bf Examples.} Here are the matrices of the pairing $\langle-,-\rangle$ restricted to $\h_{\PP}(n)$, for $n=1,2,3$.
$$\begin{array}{c|c}
&\tun\\
\hline \tun&1
\end{array} \hspace{1cm}
\begin{array}{c|c|c}
&\tdeux&\tun\tun\\
\hline \tdeux&0&1\\
\hline \tun\tun&1&2
\end{array}$$
$$\begin{array}{c|c|c|c|c|c|c}
&\ttroisdeux&\ttroisun&\ptroisun&\tdeux\tun&\tun\tdeux&\tun\tun\tun\\
\hline \ttroisdeux&0&0&0&0&0&1\\
\hline \ttroisun&0&0&0&0&1&2\\
\hline \ptroisun&0&0&0&1&0&2\\
\hline \tdeux\tun&0&0&1&1&1&3\\
\hline \tun\tdeux&0&1&0&1&1&3\\
\hline \tun\tun\tun&1&2&2&3&3&6
\end{array}$$

What is the transpose of $\prodh$ for this pairing?\\

{\bf Notations.} Let $P\in \DP$. We put $\displaystyle \Delta_\prodg(P)=\sum_{\substack{P_1,P_2\in \DP\\P_1\prodg P_2=P}} P_1\otimes P_2$.\\

{\bf Remark.} \begin{enumerate}
\item In other words, if $P=P_1\prodg \ldots \prodg P_r$ is the decomposition of $P$ into $1$-indecomposable posets, then:
$$\Delta(P)=\sum_{i=0}^r (P_1 \prodg \ldots \prodg P_i) \otimes (P_{i+1}\prodg \ldots \prodg P_r).$$
\item Moreover, $(\h_{\DP},\prodg,\Delta_\prodg)$ is an infinitesimal Hopf algebra, and the space of primitive elements for the coproduct $\Delta_\prodg$
is generated by the set of $1$-indecomposable double posets.
\end{enumerate}

\begin{prop}
For all $x,y,z \in \h_{\DP}$, $\langle x\prodh y,z\rangle=\langle x\otimes y,\Delta_\prodg(z)\rangle$.
\end{prop}

\begin{proof} We take $x=P,y=Q,z=R$ three double posets. Let $f\in S(P\prodh Q,R)$. We put $R_1=f(P)$ and $R_2=f(Q)$.
Let $i\in R_1$ and $j\in R_2$. As $f^{-1}(i)\leq_1 f^{-1}(j)$ by definition of $\prodh$, $i \leq_2 j$ in $R$.
Moreover, as $f^{-1}(i)$ and $f^{-1}(j)$ are not comparable for $\leq_2$ in $P \prodh Q$, necessarily $i$ and $j$ are not comparable for $\leq_1$ in $R$.
So $R=R_1\prodg R_2$. As a consequence, there exists a bijection:
$$\varrho:\left\{\begin{array}{rcl}
S(P\prodh Q,R)&\longrightarrow&\displaystyle\bigcup_{R_1\prodg R_2=R} S(P,R_1)\times S(Q,R_2)\\
f&\longrightarrow&(f_{\mid P},f_{\mid Q})\in S(P,f(P))\times S(Q,f(Q)).
\end{array}\right.$$
It is clearly injective. Let us show it is surjective. If $(g,h)\in S(P,R_1)\times S(Q,R_2)$, with $R=R_1\prodg R_2$,
let us consider the unique bijection $f:P\prodh Q\longrightarrow R$ such that $f_{\mid P}=g$ and $f_{\mid Q}=h$.
If $i\leq_1 j$ in $P\prodh Q$, then $i,j \in P$ or $i,j \in Q$ or $i\in P$ and $j\in Q$, so $g(i)\leq_2 g(j)$ or $h(i)\leq_2 h(j)$
or $f(i) \in R_1$ and $f(i)\in R_2$, so $f(i)\leq_2 f(j)$ in $R$.
If $f(i) \leq_1 f(j)$ in $R$, then $g(i)\leq_1 g(j)$ in $R_1$ or $h(i) \leq_1 h(j)$ in $R_2$, so $i\leq_2 j$ in $P$ or in $Q$, so $i\leq_2 j$ in $P\prodh Q$.
We proved that $f\in S(P\prodh Q,R)$. Finally:
$$\langle P\prodh Q,R\rangle=Card(S(P\prodh Q,R))=\sum_{R_1\prodg R_2=R}Card(S(P,R_1))Card(S(Q,R_2))=\langle P\otimes Q,\Delta_\prodg(R)\rangle.$$
\end{proof}

\subsection{Involution on $\DP$}

{\bf Notation.} We define the following involution:
$$ \iota: \left\{ \begin{array}{rcl}
\DP&\longrightarrow & \DP\\
(P,\leq_1,\leq_2)&\longrightarrow &(P,\leq_2,\leq_1).
\end{array}\right.$$

{\bf Examples.} For plane posets:
$$\begin{array}{rcl|crclc|rcl}
\tun&\longleftrightarrow&\tun&& \tun\tun&\longleftrightarrow&\tdeux&\\
\tun\tun\tun&\longleftrightarrow&\ttroisdeux&& \tun\tdeux&\longleftrightarrow&\ttroisun&& \tdeux\tun&\longleftrightarrow&\ptroisun\\
\tun\tun\tun\tun&\longleftrightarrow&\tquatrecinq&& \tun\tun\tdeux&\longleftrightarrow&\tquatrequatre&& \tun\tdeux\tun&\longleftrightarrow&\pquatrehuit\\
\tun\ptroisun&\longleftrightarrow&\tquatredeux&& \tun\ttroisun&\longleftrightarrow&\tquatretrois&& \tun\ttroisdeux&\longleftrightarrow&\tquatreun\\
\tdeux\tun\tun&\longleftrightarrow&\pquatrequatre&& \tdeux\tdeux&\longleftrightarrow&\pquatresept&& \ptroisun\tun&\longleftrightarrow&\pquatredeux\\
\pquatreun&\longleftrightarrow&\ttroisdeux\tun&& \pquatresix&\longleftrightarrow&\pquatrecinq&& \pquatretrois&\longleftrightarrow&\ttroisun\tun\\
\end{array}$$

\begin{prop} \label{33}
For all $P,P_1,P_2 \in \DP$:
\begin{enumerate}
\item $\iota(P_1\prodg P_2)=\iota(P_1) \prodh \iota(P_2)$ and $\iota(P_1\prodh P_2)=\iota(P_1) \prodg \iota(P_2)$.
\item $P$ is $1$-indecomposable (respectively $2$-indecomposable) if, and only if, $\iota(P)$ is $2$-indecomposable (respectively $1$-indecomposable).
\item $P$ is plane if, and only if, $\iota(P)$ is plane.
\item $P$ is WN if, and only if, $\iota(P)$ is WN.
\end{enumerate}\end{prop}

\begin{proof} $1$-$3$ are obvious. The last point comes from the fact that $\iota$ permutes $\pquatresix$ and $\pquatrecinq$. \end{proof}\\

\subsection{Non-degeneracy of the pairing $\langle-,-\rangle$}

Let $P$ be a double poset. We define:
$$\left\{\begin{array}{rcl}
X_P&=&Card(\{(x,y)\in P^2\mid x <_1 y\}),\\
Y_P&=&Card(\{(x,y)\in P^2\mid x <_2 y\}).
\end{array}\right.$$

\begin{lemma}\begin{enumerate}
\item Let $P,Q\in \DP(n)$, such that $\langle P,Q\rangle\neq 0$. Then $X_P\leq X_{\iota(Q)}$ and $Y_P\geq Y_{\iota(Q)}$.
Moreover, if $X_P=X_{\iota(Q)}$ and $Y_P=Y_{\iota(Q)}$, then $P=\iota(Q)$.
\item $S(P,\iota(P))$ is the set of automorphisms of the double poset $P$ (so is not empty).
Moreover, if $P$ is plane, then $S(P,\iota(P))$ is reduced to a single element.
\end{enumerate}\end{lemma}

\begin{proof} 1. We assume that $S(P,Q) \neq \emptyset$: let us choose $\sigma \in S(P,Q)$.
If $x<_1 y$ in $P$, then $\sigma(x)<_2 \sigma(y)$ in $Q$, so $\sigma(x)<_1 \sigma(y)$ in $\iota(Q)$. As a consequence, 
$X_P\leq X_{\iota(Q)}$. If $x<_2 y$ in $\iota(Q)$, then $x<_1 y$ in $Q$, so
$\sigma^{-1}(x)<_2 \sigma^{-1}(y)$ in $Q$. As a consequence, $Y_{\iota(Q)}\leq Y_P$.\\

Moreover, if $X_P=X_{\iota(P)}$ and $Y_P=Y_{\iota(Q)}$, then $x <_1 y$ in $P$, if, and only if $\sigma(x)<_1\sigma(y)$ in $\iota(P)$;
$x <_2 y$ in $\iota(Q)$ if, and only if, $\sigma^{-1}(x)<_2 \sigma^{-1}(y)$ in $P$. In other terms, $\sigma$ is an isomorphism of double posets
from $P$ to $\iota(Q)$, so $P=\iota(Q)$. \\

2. Let $\sigma \in S(P,\iota(P))$. If $x<_1 y$ in $P$, then $\sigma(x)<_2 \sigma(y)$ in $\iota(P)$, so $\sigma(x)<_1 \sigma(y)$ in $P$.
As $P$ is finite, this is in fact an equivalence. If $\sigma(x)<_2 \sigma(x)$ in $P$, then $\sigma(x)<_1 \sigma(x)$ in $\iota(P)$,
so $x<_2 y$. As $P$ is finite, this is an equivalence. Finally, we obtain that $\sigma$ is an automorphism of $P$.
In the other sense, if $\sigma$ is an automorphism of $P$, it is clear that $\sigma \in S(P,\iota(P))$. \\

Let us assume that $P$ is plane and let us take $\sigma \in S(P,\iota(P))$.
As $\sigma$ is an automorphism, it is increasing for $\leq_h$ and $\leq_r$, so it is also increasing for the total order $\leq$ of proposition \ref{11},
so $\sigma$ is the unique increasing bijection from $P$ to $P$ for $\leq$, that is to say $Id_P$. \end{proof}

\begin{theo}\label{37} \begin{enumerate}
\item $\langle-,-\rangle$ is non-degenerate if, and only if, the characteristic of the base field $K$ is zero.
\item $\langle-,-\rangle_{\mid \h_{\PP}}$ is non-degenerate. 
\item $\langle-,-\rangle_{\mid \h_{\WNP}}$ is non-degenerate. 
\end{enumerate} \end{theo}

\begin{proof} Let us fix $n \in \mathbb{N}$ we choose a total order on $\DP(n)$ such that, for any double posets $P,Q \in \DP(n)$:
$$((X_P,Y_P)\neq(X_Q,Y_Q),\: X_P\leq X_Q\mbox{ and } Y_P\geq Y_Q)\Longrightarrow (P \geq Q).$$
Let $P,Q \in \DP(n)$, such that $\langle P,Q \rangle \neq 0$. Then $X_P\leq X_{\iota(Q)}$ and $Y_P\geq Y_{\iota(Q)}$.
Moreover, if these inequalities are equalities, $P=\iota(Q)$; if $(X_P,Y_P)\neq (X_{\iota(Q)},Y_{\iota(Q)})$, then $P \geq \iota(Q)$
by choice of the order on $\DP(n)$. In both cases, $P \geq \iota(Q)$. \\

We index the elements of $\DP(n)$ such that $\iota(P_1)<\ldots <\iota(P_r)$. Then the matrix of $\langle-,-\rangle_{\mid \h_{\DP}(n)}$
in the bases $((\iota(P_1),\ldots,\iota(P_r))$ and $(P_1,\ldots,P_r)$ is lower triangular, with diagonal coefficients $\langle P,\iota(P)\rangle$
for $P \in \DP(n)$. So it is invertible if, and only if, $\langle P,\iota(P)\rangle$ is a non-zero element of $K$ for all $P\in \DP(n)$.
Hence, $\langle-,-\rangle$ is non-degenerate if, and only if, $\langle P,\iota(P)\rangle=Card(Aut(P))$ is a non-zero element of $K$ for all $P \in \DP$.\\

1. For all $n\in \mathbb{N}$,  $Aut(\wp_n)=\S_{\wp_n}$, so $\langle \wp_n,\iota(\wp_n)\rangle=n!$. 
Hence, $\langle-,-\rangle$ is non-degenerate if, and only if, $K$ is of characteristic zero.\\

2. As the set of plane poset is stable by $\iota$, we obtain that $\langle-,-\rangle_{\mid \h_{\PP}}$ is non-degenerate if, and only if,
$Card(Aut(P))\neq 0$ for all $P \in \PP$. As $Card(Aut(P))=1$ if $P$ is plane, this condition is statisfied. \\

3. Similar proof. \end{proof}\\

{\bf Remarks.} \begin{enumerate}
\item Note that $\h_{\DP}$ is self-dual, even if $K$ is not of characteristic zero, see corollary \ref{29}.
\item We could work over any commutative ring $R$, instead of a field $K$. Then it is possible to prove similarly that
$\langle-,-\rangle$ is non degenerate if, and only if, $\mathbb{Q} \subseteq R$.
\end{enumerate}

\section{Operad of WN double posets}

\subsection{An alternative description of free $2$-$As$ algebras}

The algebra of WN posets $\h_{\WNP}$ is given a coproduct, $\Delta$, and two products, $\prodg$ and $\prodh$.
Identifying $\h_{\WNP}$ and its dual (via the identification of the basis of WN posets with its dual basis), 
we can give $\h_{\WNP}$ another product $\star=\Delta^*$, defined by:
$$P\star Q=\sum_{R\in \WNP} n(P,Q;R) R,$$
where $n(P,Q;R)$ is the number of ideals $I$ of $R$ such that $P=R\setminus I$ and $Q=I$.

We also give it the coproduct $\Delta_{\prodg}=\prodg^*$, defined by:
$$\Delta(P)=\sum_{Q\prodg R=P}Q \otimes R.$$
Then $(\h_{\WNP},\star,\prodg,\Delta_\prodg)$ is a $2$-As Hopf algebra, that is to say:
\begin{itemize}
\item $(\h_{\WNP},\star,\Delta_\prodg)$ is a Hopf algebra. Identifying the basis $\WNP$ and its dual basis, it is the graded dual of $(\h_{\WNP},\prodg,\Delta)$.
\item $(\h_{\WNP},\prodg,\Delta_\prodg)$ is an infinitesimal Hopf algebra.  Identifying the basis $\WNP$ and its dual basis, it is self-dual.
\end{itemize}
Moreover, the space of primitive elements of $(\h_{\WNP},\Delta_\prodg)$ is generated by the set of $h$-connected WN posets $\WNP_h$.\\

{\bf Examples.} \begin{eqnarray*}
\tun \star \tdeux&=&\tun\tdeux+\tdeux \tun+2\ptroisun+\ttroisdeux\\
\tdeux \star \tun&=&\tun\tdeux+\tdeux \tun+2\ttroisun+\ttroisdeux\\
\tdeux \star \tdeux&=&2\tdeux\tdeux+\pquatresix+\pquatrecinq+\pquatredeux+\pquatretrois+2\pquatrehuit+\tquatrecinq
\end{eqnarray*}

\begin{prop}
Let $\phi:(\h_{\WNP},\prodg,\prodh)\longrightarrow (\h_{\WNP},\star,\prodg)$ be the unique morphism of $2$-$As$ algebras sending $\tun$ on $\tun$.
Then, for any WN poset $P$:
$$\phi(P)=\sum_{Q\in \WNP} \langle P,Q\rangle Q.$$
\end{prop}

\begin{proof} For any double poset $P$, there is only a finite number of double posets $Q$ such that $\langle P,Q\rangle \neq 0$
(as if it is the case, $Q$ and $P$ must have the same number of vertices). We can define a linear map:
$$\varphi:\left\{ \begin{array}{rcl}
\h_{\WNP}&\longrightarrow&\h_{\WNP}\\
P&\longrightarrow&\displaystyle \sum_{Q\in \WNP} \langle P,Q\rangle Q.
\end{array}\right.$$
It is clear that $\varphi(\tun)=\tun$. If $P_1$ and $P_2$ are two WN posets:
\begin{eqnarray*}
\varphi(P_1\prodg P_2)&=&\sum_{Q \in \WNP}\langle P_1\prodg P_2,Q\rangle Q\\
&=&\sum_{Q\in \WNP}\langle P_1 \otimes P_2,\Delta(Q)\rangle Q\\
&=&\sum_{Q_1,Q_2 \in \WNP}\langle P_1,Q_1\rangle \langle P_2,Q_2\rangle Q_1 \star Q_2\\
&=&\varphi(P_1)\star \varphi(P_2);\\ \\
\varphi(P_1\prodh P_2)&=&\sum_{Q \in \WNP}\langle P_1\prodh P_2,Q\rangle Q\\
&=&\sum_{Q\in \WNP}\langle P_1 \otimes P_2,\Delta_\prodg (Q)\rangle Q\\
&=&\sum_{Q_1,Q_2 \in \WNP}\langle P_1,Q_1\rangle \langle P_2,Q_2\rangle Q_1 \prodg Q_2\\
&=&\varphi(P_1)\prodg \varphi(P_2).
\end{eqnarray*}
So $\varphi=\phi$. \end{proof}\\

{\bf Remarks.} \begin{enumerate}
\item As a consequence: 
$$\phi\circ \iota(P)=\sum_{Q \in \WNP} n(P,Q) Q,$$
where $n(P,Q)$ is the number of bijections $f:P\longrightarrow Q$, such that $f$ is increasing for $\leq_h$ and $f^{-1}$ is increasing for $\leq_r$. 
Moreover, $\phi \circ \iota$ is the unique morphism of $2$-$As$ algebras from $(\h_{\WNP},\prodg,\prodh)$ to $(\h_{\WNP},\prodg,\star)$ sending $\tun$ to $\tun$.
\item As $(\h_{\WNP},\prodg,\prodh,\Delta)$ and $(\h_{\WNP},\star,\prodg,\Delta_\prodg)$ are two $2$-$As$ Hopf algebras, $\phi$ also satisfies the assertion
$\Delta_\prodg\circ \phi=(\phi \otimes \phi) \circ \Delta$.
\end{enumerate}

\begin{cor}
The morphism $\phi$ is bijective. As a consequence, $(\h_{\WNP},\star,\prodg)$ is freely generated, as a $2$-As algebra, by $\tun$.
\end{cor}

\begin{proof} The morphism $\phi$ is homogeneous. Let us fix an integer $n \in \mathbb{N}$.
The matrix of the restriction $\phi:(\h_{\WNP})_n\longrightarrow (\h_{\WNP})_n$ in the basis of WN posets of degree $n$ is given by 
the matrix in the same basis of the pairing $\langle-,-\rangle_{\mid (\h_{\WNP})_n}$. As the pairing $\langle-,-\rangle_{\mid \h_{\WNP}}$
is non-degenerate (theorem \ref{37}), this matrix is invertible, so $\phi$ is an isomorphism. \end{proof}

\subsection{The $B_\infty$-algebra of connected WN posets}

As a consequence, the space $Prim(\h_{\WNP})=vect(\WNP_h)$ inherits a structure $[-;-]_{m,n}$ of $B_\infty$-algebra, 
defined for all $m,n \in \mathbb{N}^*$ by:
$$\xymatrix{Prim(\h_{\WNP})^{\otimes m}\otimes Prim(\h_{\WNP})^{\otimes n}\ar[r]\ar[d]_{m_\star}&Prim(\h_{\WNP})\\
\h_{\WNP}\ar[ur]_\pi&}$$
where $\pi$ is the canonical projection on $Prim(\h_{\WNP})$ and $m_\star$ is defined by:
$$m_\star : \left\{ \begin{array}{rcl}
Prim(\h_{\WNP})^{\otimes m}\otimes Prim(\h_{\WNP})^{\otimes n}&\longrightarrow &\h_{\WNP} \\
(P_1\otimes \cdots \otimes P_m) \otimes (Q_1\otimes \cdots \otimes Q_n)&\longrightarrow &
(P_1\prodg \cdots \prodg P_m) \star (Q_1\prodg \cdots \prodg Q_n)
\end{array}\right.$$
Hence, for all $P_1,\cdots,P_m,Q_1,\cdots,Q_n \in \WNP_h$:
$$[ P_1,\cdots,P_m;Q_1,\cdots,Q_n]=\sum_{R\in \WNP_h} n(P_1\ldots P_m, Q_1\ldots Q_n;R)R.$$
For example, $[ \tun,\cdots \tun;\tun,\cdots,\tun]_{p,q}=\tun^p \prodh \tun^q$.

\begin{theo}
Let $B$ be a $B_\infty$-algebra and let $x \in B$. There exists a unique $B_\infty$-algebra morphism
$\phi:Prim(\h_{\WNP})\longrightarrow B$, sending $\tun$ to $x$.
In other terms, $Prim(\h_{\WNP})$ is the free $B_\infty$ algebra generated by $\tun$.
\end{theo}

\begin{proof} This result is proved in \cite{Loday2}. We here give a complete proof for the reader's convenience.\\

 \emph{ Existence.} By definition of a $B_\infty$-algebra, the tensor coalgebra $T(B)$ is given a structure of Hopf algebra
via the product $\star_B$, defined as the unique coalgebra morphism $\star_B:T(B)\otimes T(B) \longrightarrow T(B)$,
such that for all $m,n \in \mathbb{N}^*$, for all $x_1,\cdots,x_m,y_1,\cdots,y_n \in B$:
$$\pi_B((x_1\otimes \cdots \otimes x_m) \star_B (y_1\otimes \cdots \otimes y_n))=[ x_1,\cdots,x_m;y_1,\cdots,y_n]_B,$$
where $\pi:T(B)\longrightarrow B$ is the canonical projection.
As a consequence, denoting by $\prodg_B$ the concatenation product of $T(B)$,
$(T(B),\star_B,\prodg_B,\Delta)$ is a $2$-As Hopf algebra. As $x \in B=Prim(T(B))$, there exists a unique morphism $\psi$
of $2$-As Hopf algebra from $\h_{\WNP}$ to $B$, sending $\tun$ to $x$. We consider the diagram:
$$\xymatrix{&\h_{\WNP}\otimes \h_{\WNP} \ar[r]^{\psi \otimes \psi}\ar[d]^{B_\infty}\ar[ldd]_{\star}&T(B)\otimes T(B)\ar[d]_{B_\infty}\ar[rdd]^{\star_B}&\\
&Prim(\h_{\WNP})\ar[r]^\psi&B&\\
\h_{\WNP}\ar[ru]_\pi\ar[rrr]_\psi&&&T(B)\ar[lu]^{\pi_B}}$$
The two triangles commute; the external diagram commutes as $\psi$ is a morphism of $2$-As algebras;
the trapeze also commutes. As a consequence, the rectangle commutes, so $\psi:Prim(\h_{\WNP})\longrightarrow B$
(well-defined as $\psi$ is a morphism of coalgebras) is a morphism of $B_\infty$ algebras, sending $\tun$ to $x$.\\

\emph{ Unicity.} If $\psi'$ is another $B_\infty$ algebra morphism sending $\tun$ to $x$, then the following diagram commutes:
$$\xymatrix{&\h_{\WNP}\otimes \h_{\WNP} \ar[r]^{T(\psi') \otimes T(\psi')}\ar[d]^{B_\infty}\ar[ldd]_{\star}&
T(B)\otimes T(B)\ar[d]_{B_\infty}\ar[rdd]^{\star_B}&\\
&Prim(\h_{\WNP})\ar[r]^{\psi'}&B&\\
\h_{\WNP}\ar[ru]_\pi&&&T(B)\ar[lu]^{\pi_B}}$$
By the universal property of the coalgebra $T(B)$, there exists a unique coalgebra morphism
$\Psi'$, making the diagram commuting:
$$\xymatrix{&\h_{\WNP}\otimes \h_{\WNP} \ar[r]^{T(\psi') \otimes T(\psi')}\ar[d]^{B_\infty}\ar[ldd]_{\star}&
T(B)\otimes T(B)\ar[d]_{B_\infty}\ar[rdd]^{\star_B}&\\
&Prim(\h_{\WNP})\ar[r]^{\psi'}&B&\\
\h_{\WNP}\ar[ru]_\pi\ar[rrr]_{\Psi'}&&&T(B)\ar[lu]^{\pi_B}}$$
So $\Psi'$ is a morphism of $2$-As algebra. By the universal property of $\h_{\WNP}$ (unicity),
$\Psi'=\psi$ defined earlier. So, considering the trapeze, $\psi'=\psi_{\mid Prim(\h_{\WNP})}$. \end{proof}\\

{\bf Remark.} We can similarly describe the free $B_\infty$ algebra generated by a set $\D$, using double posets
decorated by $\D$, that is to say couples $(P,d)$, where $P$ is a double poset and $d$ is a map from $P$ to $\D$.

\subsection{A combinatorial description of the $2$-As operad}

\begin{defi}\textnormal{\begin{enumerate}
\item Let $P \in \WNP$ and let $Q \subseteq P$. We shall say that $Q$ is a \emph{ complete} subposet of $P$
if for $x,z \in Q$, $y\in P$, ($x \leq_h y\leq_h z$ $\Longrightarrow$ $y\in Q$)
and ($x \leq_r y\leq_r z$ $\Longrightarrow$ $y\in Q$). In other terms, a complete subposet is stable under intervals for $\leq_h$ and $\leq_r$.
\item Let $P$ and $Q$ be elements of $\WNP$. Let $(P_i)_{i\in Q}$ be a family of elements of $\WNP$ indexed by the elements of $Q$.
We shall say that it is a $Q$-family of $P$ if:
\begin{itemize}
\item For all $i\in Q$, $P_i$ is a complete subposet of $P$.
\item $P$ is the disjoint union of the $P_i$'s.
\item For all $i\neq j$ in $Q$, $i\leq_h j$ in $Q$ if, and only if, there exists $x_i\in P_i$, $x_j \in P_j$,
$x_i \leq_h x_j$ in $P$.
\item For all $i\neq j$ in $Q$, $i\leq_r j$ in $Q$ if, and only if, for all $x_i\in P_i$, $x_j \in P_j$,
$x_i \leq_r x_j$ in $P$.
\end{itemize}
\item We shall denote by $n_Q(P_1,\cdots,P_k;P)$ the number of $Q$-families $(P_i')_{i\in Q}$ of $P$,
such that $P'_i=P_i$ for all $i\in Q$.
\end{enumerate}}\end{defi}

{\bf Remark.} These concepts can be generalized to decorated double posets. \\

{\bf Notations.} Let $\D$ be a set. We denote by $\WNP^\D$ the set of WN posets decorated by $\D$, that is to say couples $(P,d)$,
where $P$ is a WN poset and $d:P\longrightarrow \D$ a map.

\begin{prop}
Let $(p_d)_{d\in \D}$ be a family of elements of $\WNP^{\D'}$. We consider the following map:
$$\Xi  : \left\{ \begin{array}{rcl}
(\h^\D_{\WNP},\star,\prodg)&\longrightarrow &(\h^{\D'}_{\WNP},\star,\prodg)\\
Q \in \WNP^\D&\longrightarrow &\displaystyle \sum_{P\in \WNP^{\D'}} n_{\overline{Q}}(P_{d_1},\cdots,P_{d_k};P)P,
\end{array}\right.$$
where $\overline{Q}$ is the non-decorated double poset subjacent to $Q$ and $d_i$ is the decoration of the $i$-th element of $Q$
for all $i\in Q$. Then $\Xi$ is the unique morphism of $2$-As algebra which sends $\tun_r$ on $p_r$ for all $d\in \D$.
\end{prop}

{\bf Notations.} For all $Q\in \WNP(k)$, $P_1,\cdots,P_k \in \WNP^\D$, we put:
$$\F_{P_1,\cdots,P_k}^Q=\{(P,F)\:/\: P\in \WNP^\D, 
\mbox{$F=(P'_1,\cdots,P'_k) $ is a $Q$-family of $P$ such that $P_i'=P_i$ for all $i$}\}.$$

\begin{proof} For all $d\in \D$:
$$\Xi(\tun_r)=\sum_{P\in \WNP^\D} n_{\tun}(P_r;P)P=\sum_{P\in \WNP^\D} \delta_{P_r,P}P=P_r.$$
Let $Q_1,Q_2 \in \WNP$. We denote by $d_1,\cdots,d_{k_1}$ the decorations of the elements of $Q_1$,
$d_{k_1+1},\cdots,d_{k_1+k_2}$ the decorations of the elements of $Q_2$. Then:
$$\Xi(Q_1\prodg Q_2)=\sum_{(P,F)\in \F_{P_{d_1},\cdots, P_{d_{k_1}+d_{k_2}}}^{\overline{Q_1}\prodg \overline{Q_2}}}P.$$
There is an immediate bijection:
$$ \left\{ \begin{array}{ccl}
\F_{P_{d_1},\cdots, P_{d_{k_1}}}^{\overline{Q_1}}\times \F_{P_{d_{k_1+1}},\cdots, P_{d_{k_1+k_2}}}^{\overline{Q_2}}
&\longrightarrow & \F_{P_{d_1},\cdots, P_{d_{k_1}+d_{k_2}}}^{\overline{Q_1}\prodg \overline{Q_2}} \\
((P_1,F_1),(P_2,F_2))&\longrightarrow &(P_1\prodg P_2,(F_1,F_2)).
\end{array}\right.$$
So:
$$\Xi(Q_1\prodg Q_2)=\sum_{(P_1,F_1),\:(P_2,F_2)}P_1 \prodg P_2=\Xi(Q_1)\prodg \Xi(Q_2).$$

Let us now consider $\Xi(Q_1\star Q_2)$. We put:
\begin{eqnarray*}
E_1&=&\left\{(P,I,R,F)\:/P\in \WNP^\D,\:\mbox{$I$ ideal of $P$},\: P-I=Q_1,\:I=Q_2,\:
(R,F)\in \F_{P_{d_1},\cdots,P_{d_k}}^{\overline{P}}\right\},\\
E_2&=&\left\{\begin{array}{l}
(P_1,F_1,P_2,F_2,R,I)\:/\:(P_1,F_1)\in \F_{P_{d_1},\cdots,P_{d_{k_1}}}^{\overline{Q_1}},
(P_2,F_2)\in \F_{P_{d_{k_1+1}},\cdots,P_{d_{k_1+k_2}}}^{\overline{Q_2}},\\
\:R\in \WNP,\:\mbox{$I$ ideal of $R$},\: R-I=P_1,\:I=P_2
\end{array}\right\}.\end{eqnarray*}
Then:
$$\Xi(Q_1\star Q_2)=\sum_{(P,I,R,F)\in E_1}R,\hspace{.5cm} \Xi(Q_1)\star \Xi(Q_2)=\sum_{(P_1,F_1,P_2,F_2,R,I)\in E_2} R.$$
There is a bijection from $E_1$ to $E_2$, sending $(P,I,R,F)$ to $(P_1,F_1,P_2,F_2,R,J)$ defined in the following way:
denoting $F=(P'_1,\cdots,P'_k)$, $J$ is the subposet of $R$ formed by the elements of the $P'_i$'s such that $i$ is an element of $I\subseteq P$;
$P_1=R_J$ and $F_1$ is formed by the $P'_i$'s such that $i\in P-I$; $P_2=J$ and $F_2$ is formed by the $P'_i$'s such that $i\in I$.
The only problematic point is to show that $J$ is an ideal of $R$: let $x \in J$, $y\in R$, such that $x \leq_h y$.
So $x\in P_i'$ for a certain $i\in I$ and $y \in P_j'$ for a certain $j \in P$.
By definition, $i\leq_h j$ in $P$. As $I$ is an ideal of $P$, $j\in I$, so $y\in J$.\\

As a consequence, $\Xi(Q_1\star Q_2)=\Xi(Q_1)\star \Xi(Q_2)$. So $\Xi$ is a morphism of $2$-As algebras.
As $\h_{\WNP}^\D$ is freely generated by the $\tun_r$'s, $\Xi$ is the unique $2$-As algebra morphism
which sends $\tun_r$ to $p_r$ for all $d\in \D$. \end{proof}

\begin{defi}\textnormal{\begin{enumerate}
\item For all $n \in \mathbb{N}^*$, we denote by $\WNP^{Ind}(n)$ the set of WN double posets of cardinal $n$, whose vertices are indexed, 
that is to say the set of couples $(P,d)$, where $P$ is a WN poset and  $d:P\longrightarrow \{1,\cdots,n\}$ is a bijection.
\item Let $P\in \WNP^{\mathbb{N}}$ and let $k \in \mathbb{N}$.
Then $P[k]$ is the element of $\WNP^{\mathbb{N}}$ whose subjacent double poset is $P$, and decorations obtained
from the decorations of $P$ by adding $k$.
\end{enumerate}}\end{defi}

\begin{theo}
For all $n\in \mathbb{N}^*$, we put $\P(n)=Vect\left(\WNP^{Ind}(n)\right)$. We define a structure of operad on $\P=(\P(n))_{n\in\mathbb{N}^*}$ in the following way:
for all $Q\in \WNP^{Ind}(k)$, for all $P_1,\cdots,P_k \in \WNP^{Ind}$, of respective cardinals $n_1,\cdots,n_k$, $Q\circ(P_1,\cdots,P_k)$ is $\Xi(Q)$,
where $\Xi:\h_{\WNP}^{\{1,\cdots,k\}}\longrightarrow \h_{\WNP}^{\mathbb{N}}$ is the unique morphism of $2$-As algebra which sends $\tun_1$ to $P_1$,
$\tun_2$ to $P_2[n_1]$, $\cdots$, and $\tun_k$ to $P_k[n_1+\cdots+n_{k-1}]$. The action of the symmetric group $\S_n$ on $\P(n)$ is given 
by permutation of the indices. This operad is isomorphic to the operad of $2$-As algebras.
\end{theo}

In other terms:
$$Q \circ(P_1,\cdots,P_k)=\sum_{P\in \WNP^{Ind}_{\WNP}} n_{\overline{Q}}(P'_{d_1},\cdots,P'_{d_k};P)P,$$
where $d_1,\cdots,d_k$ are the indices of the vertices of $Q$, and $P'_i=P_i[n_1+\cdots+n_{i-1}]$ for all $i$.\\

\begin{proof} Comes from the description of an operad from its free algebras. \end{proof}

\begin{cor}
For all $n\in \mathbb{N}^*$, we put $\P'(n)=Vect\left(\WNP_h^{Ind}(n) \right)$.
Then $\P'=(\P'(n))_{n\in \mathbb{N}^*}$ is a suboperad of $\P$, isomorphic to the operad of $B_\infty$-algebras.
\end{cor}

For example:
$$\tddeux{$1$}{$2$}\circ (\tdun{$1$},\tddeux{$1$}{$2$})=\pdtroisun{$3$}{$1$}{$2$}+\tdtroisdeux{$1$}{$2$}{$3$},\hspace{.5cm}
\tddeux{$1$}{$2$}\circ (\tddeux{$1$}{$2$},\tdun{$1$})=\tdtroisun{$1$}{$3$}{$2$}+\tdtroisdeux{$1$}{$2$}{$3$}.$$

The operation $\langle-;-\rangle:V^{\otimes m}\otimes V^{\otimes n}\longrightarrow V$ acting on any $B_\infty$-algebra $V$
corresponds to the element $b_{m,n}=(\tun_1\ldots \tun_m)\prodh (\tun_{m+1}\ldots \tun_{m+n})$ of $\WNP_h^{Ind}(m+n)$.
For example, $b_{1,1}=\tddeux{$1$}{$2$}$, $b_{1,2}=\tdtroisun{$1$}{$3$}{$2$}$, $b_{2,1}=\pdtroisun{$3$}{$1$}{$2$}$
and $b_{2,2}=\pdquatresept{$1$}{$2$}{$3$}{$4$}$. The Hasse graph of $b_{m,n}$ is a complete $(m,n)$ bipartite graph.

\bibliographystyle{amsplain}
\bibliography{biblio}

\providecommand{\bysame}{\leavevmode\hbox to3em{\hrulefill}\thinspace}
\providecommand{\MR}{\relax\ifhmode\unskip\space\fi MR }
\providecommand{\MRhref}[2]{%
  \href{http://www.ams.org/mathscinet-getitem?mr=#1}{#2}
}
\providecommand{\href}[2]{#2}
\begin{thebibliography}{10}

\bibitem{Foissy4}
L.~Foissy, \emph{Plane posets, special posets, and permutations}, Adv. Math.
  \textbf{240} (2013), 24--60.

\bibitem{Foissy}
Lo{\"\i}c Foissy, \emph{Les alg\`ebres de {H}opf des arbres enracin\'es
  d\'ecor\'es. {I}}, Bull. Sci. Math. \textbf{126} (2002), no.~3, 193--239,
  arXiv:math/0105212.

\bibitem{Foissy3}
\bysame, \emph{The infinitesimal {H}opf algebra and the poset of planar
  forests}, J. Algebraic Combin. \textbf{3} (2009), 277--309, arXiv:0802.0442.

\bibitem{Foissy2}
\bysame, \emph{Free and cofree {H}opf algebras}, J. Pure Appl. Algebra
  \textbf{216} (2012), no.~2, 480--494, arXiv:1010.5402.

\bibitem{Holtkamp}
Ralf Holtkamp, \emph{Comparison of {H}opf algebras on trees}, Arch. Math.
  (Basel) \textbf{80} (2003), no.~4, 368--383.

\bibitem{Loday}
Jean-Louis Loday and Mar{\'{\i}}a Ronco, \emph{On the structure of cofree
  {H}opf algebras}, J. Reine Angew. Math. \textbf{592} (2006), 123--155,
  arXiv:math/0405330.

\bibitem{Loday2}
\bysame, \emph{Combinatorial {H}opf algebras}, Quanta of maths, Clay Math.
  Proc., vol.~11, Amer. Math. Soc., Providence, RI, 2010, arXiv:0810.0435,
  pp.~347--383.

\bibitem{Reutenauer}
Claudia Malvenuto and Christophe Reutenauer, \emph{A self paired {H}opf algebra
  on double posets and a {L}ittlewood-{R}ichardson rule}, J. Combin. Theory
  Ser. A \textbf{118} (2011), no.~4, 1322--1333, arXiv:0905.3508.

\bibitem{Sloane}
N.~J.~A Sloane, \emph{On-line encyclopedia of integer sequences}, avalaible at
  http://www.research.att.com/$\sim$njas/sequences/Seis.html.

\bibitem{Stanley}
Richard~P. Stanley, \emph{Enumerative combinatorics. {V}ol. 1}, Cambridge
  Studies in Advanced Mathematics, vol.~49, Cambridge University Press,
  Cambridge, 1997.

\end{thebibliography}

\end{document}